\documentclass[a4paper,reqno]{amsart}
\usepackage{amscd,amsfonts,amssymb,amsmath,amsthm,bbm}
\usepackage{graphicx}
\usepackage[utf8]{inputenc}
\usepackage[english]{babel}

\usepackage[colorlinks=true,linkcolor=blue,citecolor=blue,urlcolor=blue,breaklinks]{hyperref}

\usepackage{mathrsfs}
\usepackage{latexsym}
\usepackage{enumerate}
\usepackage{amsopn}
\usepackage{amsbsy}
\usepackage{amscd,indentfirst,epsfig}

\usepackage{bbold} 

\usepackage{fullpage}

\hypersetup{pdftitle={Entropy dissipation estimates for the linear Boltzmann operator}}
\hypersetup{pdfauthor={M. Bisi, J. A. Ca\~{n}izo and B. Lods}}

\usepackage{breakurl}
\usepackage{natbib}
\usepackage{url}
\usepackage{color}

\theoremstyle{plain}
\newtheorem{theo}{Theorem}[section]
\newtheorem{lemme}[theo]{Lemma}

\newtheorem{propo}[theo]{Proposition}

\newtheorem{nb}[theo]{Remark}

\theoremstyle{definition}
\newtheorem{exa}[theo]{Example}

\def \leq {\leqslant}
\def \geq {\geqslant}

\numberwithin{equation}{section}
\def\ind#1{\lower5pt\hbox{$\scriptstyle #1$}}

\def \d {\,\mathrm{d}}
\def \ddt {\frac{\mathrm{d}}{\mathrm{d}t}}
\def \L {\mathcal{L}}

\def \H {\mathcal{H}}

\def \ds {\displaystyle}

\def \D {\mathcal{D}}

\def \ds {\displaystyle}
\def\Q {\mathcal{Q}}

\def\R{{\mathbb R}}

\def \S {{\mathbb S}^{d-1}}
\def \q {q}
\def \n {n}

\def \v {{v}}

\def \vb {\v_*}

\def \M {{M}}
\def \p {\partial}
\def \dis {\displaystyle}

\def\1{\mathbb{1}}

\def \It {\int_{\R^d} \int_{\S}}
\def \IS {\int_{\S}}
\def \IR {\int_{\R^d}}
\def \ird {\int_{\R^d}}
\def \ir3 {\int_{\R^3}}

\title[Entropy dissipation for the linear Boltzmann equation]
{Entropy dissipation estimates for the linear Boltzmann operator}

\author{Marzia Bisi}

\address{\textbf{Marzia Bisi}. Dipartimento di Matematica e Informatica, Universit\`{a} di Parma,
  Parco Area delle Scienze 53/A, 43124 Parma, Italy}
\email{marzia.bisi@unipr.it}

\author{Jos\'e A.~Ca\~nizo}

\address{\textbf{Jos\'{e} A.~Ca\~{n}izo}. School of Mathematics,
  University of Birmingham, Edgbaston, Birmingham B15 2TT, UK}
\email{j.a.canizo@bham.ac.uk}

\author{Bertrand Lods}
\address{\textbf{Bertrand Lods}. Universit\`{a} degli
  Studi di Torino \& Collegio Carlo Alberto, Department of Economics and
  Statistics, Corso Unione Sovietica, 218/bis, 10134 Torino, Italy.}
\email{lodsbe@gmail.com}

\begin{document}

\begin{abstract}
  We prove a linear inequality between the entropy and entropy
  dissipation functionals for the linear Boltzmann operator (with a
  Maxwellian equilibrium background). This provides a positive answer
  to the analogue of Cercignani's conjecture for this linear collision
  operator. Our result covers the physically relevant case of
  hard-spheres interactions as well as Maxwellian kernels, both with
  and without a cut-off assumption. For Maxwellian kernels, the proof
  of the inequality is surprisingly simple and relies on a general
  estimate of the entropy of the gain operator due to \cite{MT12,
    fisher}. For more general kernels, the proof relies on a
  comparison principle. Finally, we also show that in the grazing
  collision limit our results allow to recover known logarithmic
  Sobolev inequalities.
\end{abstract}

\maketitle

\tableofcontents

\section{Introduction}
\setcounter{equation}{0}

\subsection{Setting of the problem and main result}

The use of Lyapunov functionals is a well-known technique to study the
asymptotic behavior of dynamical systems, and in the theory of the
Boltzmann equation and related models it is now a classical tool. For
the nonlinear, spatially homogeneous Boltzmann equation
\begin{equation}
  \label{eq:nonlinear-Boltzmann}
  \partial_t f = \Q(f,f), \qquad f(0,v)=f_0(v),
  \qquad v \in \R^d, \,t \geq 0,
\end{equation}
posed for a function $f = f(t,v)$ depending on $t \geq 0$ and $v \in
\R^d$, it is a well-known fact that $f(t,v)$ converges (as $t\to
\infty$) towards the Maxwellian distribution $M_f$ with same mass,
momentum and energy as $f_0$,
\begin{equation*}
  M_f(v) = \dfrac{\varrho_f}{(2\pi\,E_f)^{d/2}}
  \exp\left(-\frac{|v-\mathbf{u}_f|^2}{2\,E_f}\right),
  \qquad v \in \R^d,
\end{equation*}
where
\begin{equation*}
  \left.
  \begin{aligned}
    \varrho_f &= \int_{\R^d} f(t,v)\d v = \int_{\R^d} f_0(v)\d v,
    \\
    \varrho_f \mathbf{u}_f &= \int_{\R^d} f(t,v) v\d v=\int_{\R^d}
    f_0(v)v \d v,
    \\
    \quad d\, \varrho_f E_f &= \int_{\R^d} f(t,v)|v-\mathbf{u}_f|^2\d
    v=\int_{\R^d} f_0(v)\,|v-\mathbf{u}_f|^2\d v
  \end{aligned}
  \quad
  \right\}
  \qquad \text{ for all $t \geq 0$.}
\end{equation*}
Notice that eq.~\eqref{eq:nonlinear-Boltzmann} conserves density,
momentum and kinetic energy which explains why the above quantities
$\varrho_f$, $\mathbf{u}_f$ and $E_f$ are constant in time.
The \emph{Shannon-Boltzmann relative entropy} of $f$ with respect to
the Maxwellian distribution $M_f$
\begin{equation}
  \label{eq:relative-entropy}
  \H(f|M_f) := \ird f(v) \log \frac{f(v)}{M_f(v)} \d v
\end{equation}
is a Lyapunov functional, that is, it is decreasing along solutions to
\eqref{eq:nonlinear-Boltzmann}: if $f = f(t,v)$ solves
\eqref{eq:nonlinear-Boltzmann},
\begin{equation}
  \label{eq:D-nl-Boltzmann}
  \ddt \H(f|M_f) = -\D(f) \leq 0,
\end{equation}
where the functional $\D$ is called the \emph{entropy dissipation}.
The question of whether one can find a functional inequality between
$\H$ and $\D$ of the form
\begin{equation*}
  \D(f) \geq \lambda \,\mathbf{\Phi}(\H(f|M_f))
\end{equation*}
valid for some $\lambda > 0$, some nondecreasing continuous function
$\mathbf{\Phi}\;:\;[0,+\infty) \to [0,+\infty)$ with $\mathbf{\Phi}(0)=0$, and all functions
$f$ (with $f$ possibly satisfying some additional suitable bounds), is
generally known as \emph{Cercignani's conjecture}. It has several
variants and a long history (e.g. \citet{Carlen94,ToVi,Vil03}; see the
recent review by \cite{DMV11} for further details). If true, this
inequality gives a lot of information on the asymptotic behavior of
\eqref{eq:nonlinear-Boltzmann}, since then one obtains the
differential inequality
\begin{equation*}
  \ddt \H(f(t)|M_f) \leq -\lambda \, \mathbf{\Phi}(\H(f(t)|M_f)),
\end{equation*}
from which one can deduce that $\H(f(t)|M_f)$ converges to $0$ as
$t \to +\infty$, with an explicit rate. Notice that due to the
Csisz\'ar-Kullback-Pinsker inequality the convergence of
$\H(f(t)|M_f)$ towards $0$ implies the convergence in $L^1(\R^d)$ of
$f(t,v)$ towards $M_f$. Unfortunately, the available versions of
Cercignani's conjecture do not yield an optimal rate of convergence of
$f(t,v)$ towards $M_f$. However, the use of this Lyapunov functional
approach combined with a careful spectral analysis of the linearized
Boltzmann operator allow to recover an exponential convergence to
equilibrium \citep{M06}.  \medskip

We are interested in studying the corresponding conjecture in the case
of the \emph{linear} Boltzmann equation which, though simpler, has not
yet been settled.  Let us describe the model in more detail before
explaining our results. The homogeneous, linear Boltzmann equation is
given by
\begin{equation}
  \label{eq:linear-Boltzmann}
  \partial_t f = \Q(f,\M) = \L f,
  \qquad f(0,v) = f_0(v),
  \qquad t \geq 0,\, v \in \R^d,
\end{equation}
where $\Q$ is the bilinear Boltzmann operator,
\begin{equation}
  \label{bolt}
  \Q(f,g)
  =
  \It B(|q|, \xi) \bigg(
    f(\v')g(\vb') - f(\v)g(\vb)
  \bigg)\d\vb \d\n.
\end{equation}
Here $\q=\v-\vb$ is the relative velocity, $\xi = |q \cdot n| / |q|$,
and $\v'$ and $\vb'$ are the pre-collisional velocities which result,
respectively, in $ \v $ and $\vb$ after the elastic collision
\begin{equation}
    \label{vprime}
  \v'=v-(\q \cdot \n)\n, \qquad \vb'=\vb+(\q \cdot \n)\n.
\end{equation}
The particle distributions $f$ and $g$ are nonnegative functions of
the velocity variable $\v \in\R^d$ and $B(|q|, \xi)$ is a nonnegative
function usually called the \emph{collision kernel}. We will assume
throughout this paper that the function $\M$ appearing in
\eqref{eq:linear-Boltzmann} is a given normalized Maxwellian
distribution with unit mass:
\begin{equation}\label{maxwe1}
  \M(\v)=\bigg({2\pi \theta}\bigg)^{-d/2}\exp
  \left(-\dfrac{|\v-u_0|^2}{2\theta}\right), \qquad \qquad \v
  \in \R^d,
\end{equation}
where $u_0 \in \R^d$ is the \emph{bulk velocity} and $\theta > 0$ is
the \emph{effective temperature}. We notice that the normalization of
$\M$ is not a loss of generality since a time scaling of
\eqref{eq:linear-Boltzmann} easily translates into results for
non-normalized Maxwellians. Similarly, since
\eqref{eq:linear-Boltzmann} is linear, for simplicity we will assume
throughout that the solution $f$ also has mass $1$:
$$\int_{\R^{d}}f(t,v)\d v=\int_{\R^{d}}f_{0}(v)\d v=1, \qquad t \geq 0.$$ Galilean
invariance and a scaling in $v$ also easily show that one may study
only the case $\theta = 1$, $u_0 = 0$. However, we will state all results
for \eqref{maxwe1} in order to make clear how inequalities depend on
them.

We shall investigate in this paper collision operators $\L=\L_B$
corresponding to various collision kernels $B=B(|q|,\xi)$ but shall
most often deal with kernels that factor as
\begin{equation}
  \label{eq:Bqxi}
  B(|q|,\xi)=\beta(|q|)  \,b(\xi)
\end{equation}
for some measurable nonnegative mappings $b:[0,1] \to [0,\infty)$ and
$\beta(\cdot) : [0,\infty) \to [0,\infty)$. For the purposes of proofs
we always work with the cut-off assumption that
\begin{equation}
  \label{eq:cutoff}
  \IS b(\tilde{q} \cdot n) \d n < +\infty,
\end{equation}
(where $\tilde{q}= q/|q|$), though our results apply also to
non-cutoff kernels (just because the entropy dissipation is larger in
that case; see Remark \ref{rem:non-cutoff}). We always deal with
\textit{hard potential interactions}, that is, collision kernels with
$\beta$ nondecreasing.\footnote{Notice that, for collision kernel of
  the above shape, if $\beta(\cdot)$ is such that
  $\liminf_{r \to \infty}\beta(r)=0$, the convergence towards
  equilibrium is not expected to be exponential (for instance, the
  spectrum of $\L$ in the $L^2$ space with weight $\M^{-1}$ does not
  have a spectral gap)} In particular, we will deal with
\begin{equation}
  \label{eq:hard-potential-kernel}
  B(|q|,\xi) = c_d |q|^\gamma \xi^{d-2},
\end{equation}
for $\gamma \geq 0$ (with $c_d$ a normalization constant). In dimension $d=3$, the case
$\gamma = 1$ is the case of \textit{hard-spheres interactions}, while
the $\gamma = 0$ corresponds to the \textit{Maxwell molecules
  interaction}, that is,
\begin{align}
  \label{eq:hard-sphere-kernel}
  &B(|q|,\xi) = B_{\mathrm{hs}}(|q|,\xi)
  = c_d |\q \cdot \n| = c_d |q| \xi
  \qquad &&\text{(Hard-spheres),}
  \\
  \label{eq:Maxwell-kernel}
  &B(|q|,\xi) =
  B_{\mathrm{max}}(|q|,\xi) =
  c_d \frac{|q \cdot \n|}{|q|} = c_d \xi
  \qquad &&\text{(Maxwell molecules).}
\end{align}
We will also deal with general \emph{Maxwellian collision kernels},
that is, kernels which depend only on $\xi$:
\begin{equation}
  \label{eq:maxwellian-kernel}
  B(|q|,\xi) = b(\xi)
\end{equation}
for some measurable function $b:[0,1] \to [0,+\infty)$.  (The Maxwell
molecules approximation \eqref{eq:Maxwell-kernel} being a particular
case.) We say a Maxwellian collision kernel is \emph{normalized} when,
for any $\tilde{q} \in \S$,
\begin{equation}
  \label{eq:normalized}
  \IS b(\tilde{q} \cdot n) \d n
  = |\S| \int_{0}^1 b(\xi) (1-\xi^2)^{\frac{d-3}{2}}\d \xi = 1,
\end{equation}
where $|\S|$ represents the $(d-1)$-dimensional volume of $\S$.

Equation \eqref{eq:linear-Boltzmann} is sometimes known also as the
\emph{scattering} equation, and can be interpreted as giving the time
evolution of the velocity distribution of a cloud of particles,
homogeneously distributed in space. These particles do not interact
among themselves, but only with background particles whose
distribution is given by $\M$, considered as a thermal bath in the
sense that it remains unchanged even after interaction with the cloud
of particles (this is reasonable if, for example, the total mass of
the cloud of particles is much smaller than that of the background).
Equation \eqref{eq:linear-Boltzmann} conserves density (i.e., $\ird
f(t,v) \d v = \ird f(0,v) \d v$ for all $t$), but in contrast with the
nonlinear Boltzmann equation, momentum and kinetic energy are not
conserved due to the interaction with the background. Except in the
special case of \emph{Maxwellian molecules}, no explicitly solvable
differential equations can be derived for the evolution of the
momentum and the kinetic energy. (For the explicit time evolution of
momentum, energy and temperature in the case of Maxwell molecules see
for example \citet{ST04}.)

It will be sometimes convenient to express the collision operator $\L$
in the following weak form:
\begin{equation}
  \label{weak}
  \ird \psi(v) \L f(v) \d v
  =
  \int _{\R^d} \int_{\R^d} \int_{\S} B(|q|,\xi)
  f(\v){\M}(\vb)\big(\psi(\v')-\psi(\v)\big)\d
  v \d \vb \d \n
\end{equation}
for any sufficiently regular $\psi$. On the other hand, $\L$ can also
be written in the form
\begin{equation}
  \label{eq:linear-Boltzmann-kernel-form}
  \L f(v) = \IR k_B(w,v)f(w)\d w -\sigma_B(v)f(v),
  \qquad v \in \R^d
\end{equation}
for a kernel $k_B(v,w) \geq 0$ which depends of course on the
collision kernel $B$ (see for instance \cite{Carleman1957}), and with
\begin{equation*}
  \sigma_B(v) = \int_{\R^d} k_B(v,w) \d w,
  \qquad v \in \R^d.
\end{equation*}
The kernel $k_B$ can be written explicitly in some cases; see
e.g. \citet{ArLo}. For a general expression of $k_B$ see the
discussion leading to equation \eqref{eq:kB}. One sees then that
eq.~\eqref{eq:linear-Boltzmann} is the Kolmogorov forward equation for
a Markov process on $\R^d$ with invariant measure, or equilibrium,
$\M$ (notice that $\L(\M) = \Q(\M,\M) = 0$ regardless of the collision
kernel $B$), and it is well known that the relative entropy
\eqref{eq:relative-entropy} with respect to the equilibrium is a
Lyapunov functional for any equation of this type (see for example
\citet{Chafai04} or \citet{Michel}.) In addition, $\L$ satisfies the
\emph{detailed balance} condition, that is,
\begin{equation}
  \label{eq:detailed-balance}
  \M(v)k_B(v,w) = \M(w)k_B(w,v),
  \qquad v,w \in \R^d,
\end{equation}
which translates to the fact that $\L$ is symmetric in $L^2(\R^d,
M(v)^{-1} \d v)$.  Using this, we can explicitly write the time
derivative of $\H(f|\M)$ along solutions to
\eqref{eq:linear-Boltzmann}:
\begin{equation}
  \label{eq:dtH-linear-Boltz}
  \ddt \H(f(t)|\M)
  =
  \ird \L f(t,v) \log\left(\frac{f(t,v)}{\M(v)}\right) \d v
  = -\D(f(t))
\end{equation}
where the \emph{entropy dissipation} $\D(f)$ is
\begin{equation}
  \label{eq:D-linear-Boltz}
  \D(f) := \frac{1}{2} \ird \ird \int_{\S}
  B(|q|,\xi) \M(v)\M(v_*)
  \Psi\left( \frac{f(v)}{\M(v)}, \frac{f(v')}{\M(v')} \right)
  \d n \d v_* \d v,
\end{equation}
with $\Psi(x,y) := (x-y)(\log x - \log y) \geq 0$. Alternatively, we
can also write
\begin{equation}
  \label{eq:D-linear-Boltz-2}
  \D(f) := \frac{1}{2} \ird \ird
  \M(v) k_B(v,v')
  \Psi\left( \frac{f(v)}{\M(v)}, \frac{f(v')}{\M(v')} \right)
  \d v \d v',
\end{equation}
where $k_B$ is the kernel appearing in
\eqref{eq:linear-Boltzmann-kernel-form}.

It is interesting then to look for inequalities of the form $\D(f) \geq
\lambda \H(f|\M)$, for some $\lambda > 0$, since clearly this implies
that any solution $f$ to \eqref{eq:linear-Boltzmann} with mass $1$
satisfies
\begin{equation*}
  \H(f(t)|\M) \leq \H(f_0|\M) \, \exp\left(-\lambda\,t\right)
  \qquad \forall t \geq 0
\end{equation*}
yielding exponential convergence to the equilibrium $\M$ in the
entropic sense (notice that our Maxwellian $\M$ was also normalized to
have mass $1$). The following is our main result regarding this:
\begin{theo}
  \label{main}
  Let $\D$ be the entropy dissipation functional
  \eqref{eq:D-linear-Boltz} and consider either a hard-potential collision
  kernel $B$ of the form \eqref{eq:hard-potential-kernel} with $\gamma
  > 0$, or any normalized Maxwellian collision kernel (i.e.,
  satisfying \eqref{eq:maxwellian-kernel} and \eqref{eq:normalized}).
  There exists a positive constant $\lambda=\lambda(B) > 0$ such that
  \begin{equation}
    \label{entropydiss}
    \D(f) \geq \lambda \, \H(f|\M)
  \end{equation}
  holds for any probability distribution $f \in L^1(\R^d)$.

  If the collision kernel is Maxwellian then one may take
  \begin{equation}
    \label{eq:gammab}
    \lambda
    =
    \gamma_b
    := \int_{\S} (\tilde{q} \cdot n)^2 b(\tilde{q} \cdot n) \d n
    = |\S| \int_{0}^1 \xi^2 b(\xi) (1-\xi^2)^{\frac{d-3}{2}}\d \xi \in (0,1),
    \qquad \tilde{q} \in \S.
  \end{equation}
 Notice that the value of $\gamma_b$ does not depend on $\tilde{q}$
  due to radial symmetry, and is a number strictly between $0$ and $1$
  due to normalization.
\end{theo}
\begin{nb}
  \label{rem:non-cutoff}
  Notice that our results actually cover non cut-off kernels for which
  \eqref{eq:cutoff} is not satisfied, as long as they can be bounded
  below by a collision kernel to which the above theorem
  applies. Indeed, if $b\::\:[0,1] \to \R^{+}$ is such that
  $\IS b(\tilde{q} \cdot n) \d n = +\infty$ (remember this integral
  does not depend on $\tilde{q} \in \S$) then removing the
  singularities, we can bound $b$ from below by some measurable
  $b_{0}\::\:[0,1] \to \R^{+}$ satisfying
  $$\IS b_{0}(\tilde{q} \cdot n) \d n < +\infty, \qquad \tilde{q} \in \S.$$
  Since, as one sees from \eqref{eq:D-linear-Boltz}, the \emph{entropy
    dissipation} functional is monotone with respect to the collision
  kernel (while the relative entropy is obviously independent of the
  collision mechanism!), the result obtained for the cut-off kernel
  $b_{0}$ applies to the original kernel $b$. It is however likely
  that the obtained bound is far from being optimal.
\end{nb}
Also, for the hard spheres kernel \eqref{eq:hard-sphere-kernel} in
dimension $d=3$ we may take $\lambda = \sqrt{\theta}/4$ (see
Example \ref{exa:hard-spheres}). In fact, we are able to give a
general condition on $B$ ensuring that inequality \eqref{entropydiss}
holds; see Theorem \ref{thm:main-general}. This inequality is part of
a larger family of inequalities relating other Lyapunov functionals of
\eqref{eq:linear-Boltzmann} to their dissipations (see section
\ref{sec:comparison}), of which a prominent example is the
\emph{spectral gap inequality}
\begin{equation}
  \label{eq:spectral-gap-ineq}
  \ird \ird \int_{\S}
  B(|q|,\xi) \M(v)\M(v_*)
  \left( \frac{f(v)}{\M(v)}- \frac{f(v')}{\M(v')}\right)^2
  \d n \d v_* \d v
  \geq
  \frac{\lambda_2}{2}
  \ird \M \left( \frac{f}{\M} - 1 \right)^2 \d v.
\end{equation}
This inequality was already studied in \citet{LoMo}, and it implies
exponential relaxation to equilibrium in the $L^2$ norm with weight
$\M^{-1}$ for equation \eqref{eq:linear-Boltzmann}. However, it gives
a different information from inequality \eqref{entropydiss}, since
convergence is given in a different distance: the spectral gap result
gives convergence in a stronger topology, but also requires the
initial condition to have stronger decay for large $v$. Of course, the
best possible constants $\lambda_2$ and $\lambda$ may be different as
well (always with $\lambda \leq \lambda_2/2$, see \citet{Ane,
  Bakry2014}) implying different exponential relaxation speeds.

\subsection{Link to Logarithmic Sobolev inequalities}
In addition to being fundamental in the study of the asymptotic
behavior of \eqref{eq:linear-Boltzmann}, entropy dissipation
inequalities of the form of \eqref{entropydiss} have interesting links
to results in the theory of Markov processes and have been the subject
of several recent studies in discrete settings. In the framework of
discrete, time-continuous Markov processes the study of inequalities
such as \eqref{entropydiss} is relatively recent. They are often
referred to as a type of ``\textit{modified logarithmic Sobolev
  inequalities}'' in this context; see \citet{BT06, Bakry2014} for recent results
and a summary of related literature. Comparatively, entropy
dissipation inequalities for continuous-space processes have been
little studied, so it is interesting to see whether more general
techniques can be developed for them. The idea of studying the
convergence to equilibrium of a Markov process in terms of the
relative entropy to the invariant measure is in fact much older, but
it has usually been done by means of logarithmic Sobolev inequalities
instead of \eqref{entropydiss}. For our linear operator $\L$, this
would be an inequality of the form
\begin{equation}
  \label{eq:log-sob-L}
  \mathscr{E}\left( \sqrt{\M} \sqrt{f} \right)
  \geq \lambda_0 H(f|\M)
\end{equation}
for some $\lambda_0 > 0$ and all probability distributions $f$ in
$\R^d$, where $\mathscr{E}$ is the Dirichlet  form associated to $\L$:
\begin{equation*}
  \mathscr{E} (g) := -\ird   g(v) \L g(v) \,M^{-1}(v)\d v.
\end{equation*}
This approach is followed, for example, in \citet{DS96}. Though it is
written there for discrete models, one can easily follow the same
arguments here in order to see that \eqref{eq:log-sob-L} would imply
\eqref{entropydiss} with $\lambda = \lambda_0$. The interesting
problem with this approach is that for our continuous model the
logarithmic Sobolev inequality \eqref{eq:log-sob-L} \emph{cannot
  hold}. The reason for this is that, as is well-known
\citep{Gross1975Logarithmic, Gross1993Logarithmic}, the log-Sobolev
inequality \eqref{eq:log-sob-L} is equivalent to an $L^q-L^p$
regularizing property of solutions of equation
\eqref{eq:linear-Boltzmann}, known as \textit{hypercontractivity}
which does not hold for solutions to \eqref{eq:linear-Boltzmann} (see
a quick proof of this fact in Appendix \ref{sec:no-log-sob}). Hence we
have that
\begin{theo}\label{theo12}
  Under the cut-off assumption \eqref{eq:cutoff}, there is no constant
  $\lambda_0 > 0$ such that inequality \eqref{eq:log-sob-L} holds for
  all probability distributions $f$ in $\R^d$.
\end{theo}

Hence, the linear Boltzmann operator is an interesting case in which
the entropy dissipation (or modified log-Sobolev) inequality
\eqref{entropydiss} holds, but the log-Sobolev inequality
\eqref{eq:log-sob-L} does not!

The links between our modified log-Sobolev inequality
\eqref{entropydiss} and true log-Sobolev inequalities turn out to be
tighter than expected. Recall that the well-known Gaussian log-Sobolev
inequality (also known as Stam-Gross inequality) asserts that
\begin{equation}\label{eq:Gaussian-log-Sob}
I(f|\M) \geq \frac{2}{\theta}\, H(f|M)
\end{equation}
for any $f \in L^1(\R^d)$ with unit mass. Here $I(f|\M)$ is the
relative Fisher information
$$I(f|\M)=\ird  f(v) \left|\nabla \log\left(\frac{f(v)}{\M(v)}\right)\right|^2 \d v.$$
The above functional inequality is, as well-known, the entropy-entropy
dissipation estimate for the Fokker-Planck equation
\begin{equation}\label{eq:FP}\partial_t \varrho(t,v)=\nabla \cdot \left(\nabla \varrho(t,v) - \frac{\nabla \M(v)}{\M(v)} \varrho(t,v)\right)\end{equation}
since the time derivative of $H(\varrho|\M)$ along solutions to
\eqref{eq:FP} exactly yields
$$\ddt H(\varrho(t)|\M)=-I(\varrho(t)|\M) \qquad \forall t \geq 0.$$
In Proposition \ref{prp:D-comparison} below we are able to show that
inequality \eqref{entropydiss} also holds for the following family of
collision kernels (depending on $\epsilon \in (0,1]$):
\begin{equation}
\label{grazing}
  B_\epsilon(|q|,\xi) = |q| b_\epsilon(\xi),
  \qquad
  b_\epsilon(\xi) = \xi \1_{[0,\epsilon]}(\xi),
\end{equation}
where $\1_{[0,\epsilon]}$ denotes the characteristic of the interval
$[0,\epsilon]$. As $\epsilon \to 0$, a suitable scaling of equation
\eqref{eq:linear-Boltzmann} with this collision kernel approaches a
Fokker-Planck equation (with a diffusion matrix different from the
identity; see \citet{LoTo}.) The dependence of $\lambda$ on $\epsilon$
actually enables us to recover in the limit $\epsilon \to 0$ a version
of \eqref{eq:Gaussian-log-Sob} for that diffusion matrix (notice that
\eqref{eq:Gaussian-log-Sob} corresponds to the case of an identity
diffusion matrix); details of this are given in Section
\ref{sec:diffusive}. This procedure can be understood as a microscopic
validation of well-known logarithmic Sobolev inequalities.


There are interesting similarities between this result and one derived
in \citet{BT06}: it is shown there that one may obtain
\eqref{eq:Gaussian-log-Sob} as the limit of certain discrete modified
log-Sobolev inequalities. We show a similar result here, but through a
completely different limiting process.

\subsection{Method of proof}

Our proof of \eqref{entropydiss} consists in first proving the result
for the dissipation $D_{\mathrm{max}}$ of the linear Boltzmann
operator $\L_\mathrm{max}$ associated with a Maxwellian collision
kernel and then deducing the result for other collision kernels by a
comparison argument. Namely, one of the main steps in our proof is the
following comparison result whose proof closely follows the lines of a
similar result proved for the study of the spectral gap of
$\L_{\mathrm{hs}}$ \cite[Proposition 3.3]{LoMo}:
\begin{propo}
  \label{prp:D-comparison-intro}
  Take $\gamma \geq 0$ and let $\D_{\gamma}$ denote the entropy
  dissipation functional of the linear Boltzmann operator associated
  to the collision potential \eqref{eq:hard-potential-kernel} (so that
  $\gamma=0$ corresponds to Maxwellian molecules interactions). There
  is some positive explicit constant $C>0$, depending only on
  $\gamma$, such that
  $$\D_{\gamma}(f) \geq C \theta^{\gamma/2} \D_{0}(f)$$
  for any probability distribution $f$.
\end{propo}
For the proof of this (in a more general statement that allows for
comparing dissipations of other Lyapunov functionals) see Proposition
\ref{prp:D-comparison-hardpotentials}.

\medskip Then one sees that in order to prove Theorem \ref{main} it is
enough to prove it for a normalized Maxwellian collision kernel. This
lends itself to significant simplification since, as is well-known,
Maxwellian collision kernels generally allow for explicit
computations. Here is the heart of the argument, which we give for
simplicity in the Maxwellian molecules case (i.e., for $B$ given by
\eqref{eq:Maxwell-kernel}). In this case, the linear Boltzmann
operator $\L_{\mathrm{max}}$ can be written as
$$\L_{\mathrm{max}}(f)=\L^+_{\mathrm{max}}(f) -f.$$
Now, the operator $\L^+_{\mathrm{max}}(f) = \Q^+_{\mathrm{max}}(f,M)$
satisfies the following analog of the Shannon-Stam inequality
\citep[Corollary 4.3]{fisher}: for any probability densities $f$ and
$g$ it holds that
\begin{equation}
  \label{eq:villani}
  H(\Q^+_{\mathrm{max}}(f,g)) \leq \frac{1}{2} H(f)+ \frac{1}{2} H(g),
\end{equation}
where $H$ is the \emph{Shannon-Boltzmann entropy}
\begin{equation}
  \label{eq:H}
  H(f) := \ird f(v) \log f(v) \d v,
\end{equation}
defined for any nonnegative $f \in L^1(\R^d)$ with finite energy. We
show in Lemma \ref{lem:max} that this translates to a contraction
property of $\L^+_{\mathrm{max}}$, measured in entropy:
\begin{equation*}
  \H(\L^+_{\mathrm{max}} f |\M)
  \leq
  \frac 1 2 \H(f|\M).
\end{equation*}
This allows us to write
\begin{multline*}
  \D_\mathrm{max}(f)=\IR f
  \log\left(\dfrac{f}{\M}\right)\d\v
  -\IR \L^+_{\mathrm{max}}(f)\log\left(\dfrac{f}{\M}\right)\d\v
  \\
  =
  \H(f|\M)
  -\H(\L^+_{\mathrm{max}} f  | \M)
  +\H(\L^+_{\mathrm{max}} f  | f)
  \geq
  \frac 1 2 \H(f|\M).
\end{multline*}
Notice that we estimated $\H(\L^+_{\mathrm{max}} f | f) \geq 0$
since $\L^+_{\mathrm{max}}(f)$ and $f$ have the same mass. This shows
the inequality.

This provides an interesting link between the entropy dissipation
inequality \eqref{entropydiss} and the convexity property
\eqref{eq:villani} of the gain part $\Q^+_{\mathrm{max}}$ of the
bilinear Boltzmann operator. The proof of \eqref{eq:villani} was based
on a similar contraction property of $\Q^+_{\mathrm{max}}(f,g)$ with
respect to the Fisher information, along with a representation of the
Fisher information as the time-derivative of the entropy along the
adjoint Ornstein-Uhlenbeck semigroup (see equation \eqref{eq:3} in
Section \ref{sec:maxwell}); we refer to \cite{fisher} for a detailed
proof. Estimates for other Maxwellian kernels (i.e., depending only on
the $\xi$ variable) may be obtained by using extensions of
\eqref{eq:villani} which were essentially proved in \citet*{MT12}. We
refer to Section \ref{sec:maxwell} for details on this.

\subsection{Structure of the paper}

The plan of the paper is as follows. In Section \ref{sec:maxwell} we
prove our results for Maxwellian kernels (including the proof of
Theorem \ref{main} for Maxwellian kernels.) In Section
\ref{sec:non-maxwell} we prove a more general version of the
comparison result in Proposition \ref{prp:D-comparison-intro} in order
to deduce Theorem \ref{main} for hard potential interactions, thus
completing the proof of Theorem \ref{main}.  In Section
\ref{sec:speed1} we show how the entropy dissipation inequality may be
used to give an exponential rate of convergence to equilibrium for
eq.~\eqref{eq:linear-Boltzmann} (which is straightforward) and for a
nonlinear Boltzmann equation with particles bath. Finally, we describe
in Section \ref{sec:diffusive} the link between our inequality
\eqref{entropydiss} and logarithmic Sobolev inequalities. In
particular, we recall the Fokker-Planck limit of grazing collisions
and some well-known features of log-Sobolev inequalities, and then
show how some of them can be recovered from \eqref{entropydiss}.

\section{Inequalities for Maxwellian collision kernels}
\label{sec:maxwell}

We begin in this section with a proof of the following entropy
dissipation inequality for the linear Boltzmann operator with a
Maxwellian collision kernel:
\begin{theo}
  \label{theo:max}
  Let $B(|q|,\xi)=b(\xi)$ be a normalized Maxwellian
  collision kernel. Let $\D_\mathrm{max}$ denote the associated entropy
  dissipation functional. For any probability distribution $f=f(v)$
  one has
  \begin{equation}
  \label{teoMax}
    \D_\mathrm{max}(f) \geq \gamma_b \, \H(f|\M)
  \end{equation}
  with $\gamma_b$ defined in (\ref{eq:gammab}).
\end{theo}
\medskip

In order to prove this we need several previous results; the
proof of Theorem \ref{theo:max} is given at the end of this
section. We first prove the following contraction property of the entropy which
is essentially contained in \citet{MT12}:
\begin{lemme}
  \label{lem:MT}
  Let $B(|q|,\xi)=b(\xi)$ be a normalized Maxwellian
  collision kernel. For any probability distributions $f, g$ one has
  \begin{equation}
    \label{eq:MT}
    H(\Q_+(f,g))
    \leq
    (1-\gamma_b) H(f) + \gamma_b H(g)
  \end{equation}
where we recall that $H(\cdot)$ denotes the Shannon-Boltzmann entropy defined in \eqref{eq:H}.
\end{lemme}
\begin{proof}
  In \citet[eq.~(3)]{MT12} it is proved that
  for the Fisher information
  \begin{equation}
  \label{Fisher}
  I(f) = \ird \frac{|\nabla f(v)|^2}{f(v)}\, \d v
  \end{equation}
  an analogous inequality holds
  \begin{equation}
    \label{eq:2}
    I(\Q_+(f,g))
    \leq
    (1-\gamma_b) I(f) + \gamma_b I(g).
  \end{equation}
  (Notice that the estimates in \citet*{MT12} are written in terms of
  the $\sigma$-representation for Boltzmann's operator; here we have
  written the corresponding expression in the $n$-representation by a
  change of variables.) To deduce \eqref{eq:MT} from \eqref{eq:2}, we
  use a well-known strategy already used in \citet{fisher}, based on
  the nice property that the Boltzmann operator commutes with the
  adjoint Ornstein-Uhlenbeck semigroup. Namely, given a probability
  measure $f$, let $\mathcal{S}_t f(v)=\varrho(t,v)$ denote the unique solution
  (at time $t \geq 0$) to the Fokker-Planck equation \eqref{eq:FP}
  with initial datum $\varrho(0)=f$ (i.e. $(\mathcal{S}_t)_{t \geq 0}$ is the
  adjoint Ornstein-Uhlenbeck semigroup).  Whenever $B(|q|,\xi)=b(\xi)$
  is a normalized Maxwellian collision kernel we have \citep{bob}
  \begin{equation}
    \label{eq:QSt}
    \Q_+(\mathcal{S}_t f, \mathcal{S}_tf) = \mathcal{S}_t\Q_+(f, f ) \qquad \forall t \geq 0,
  \end{equation}
  for any $f \in L^1(\R^d)$ with finite energy. Moreover, a well-known
  property of Fisher information is that
  \begin{equation}    \label{eq:4}
    H(f) - H(\M)
    = \int_0^\infty (I(\mathcal{S}_t f) - I(\M)) \d t,
  \end{equation}
  for any $f \in L^1(\R^d)$ with unit mass and finite
  energy. Combining this with the above commutation property
  \eqref{eq:QSt}, one gets the following representation formula:
  \begin{equation}
    \label{eq:3}
    H(\Q_+(f,g)) - H(\M)
    = \int_0^\infty (I(\Q_+(\mathcal{S}_t f, \mathcal{S}_t g) - I(\M)) \d t.
  \end{equation}
  Applying \eqref{eq:2} in
  \eqref{eq:3} and then \eqref{eq:4} gives
  \begin{equation*}\begin{split}
    H(\Q_+(f,g)) - &H(\M)
    \\
    &\leq
    (1-\gamma_b) \int_0^\infty (I(\mathcal{S}_t f) - I(\M)) \d t
    + \gamma_b \int_0^\infty (I(\mathcal{S}_t g) - I(\M)) \d t
    \\
    &=
    (1-\gamma_b) H(f) + \gamma_b H(g)-H(\M)
  \end{split}\end{equation*}
  which completes the proof.
\end{proof}

\begin{nb}
  The proof of \eqref{Fisher} as derived in \cite{MT12} is based on an
  explicit representation of $\Q$ in Fourier variables. It is for this
  reason that it is crucial for their techniques to deal with
  Maxwellian collision kernels.
\end{nb}
We define the gain part of the linear operator $\L$ by $\L_+(f) =
\Q_+(f,\M)$. Next we show that $\L_+$ takes a function closer to the
equilibrium in the relative entropy sense.
\begin{lemme}
  \label{lem:max}
  Let $B(|q|,\xi)$ be a normalized Maxwellian collision kernel and let
  $\L$ be associated linear Boltzmann operator. Then,
  \begin{equation}
    \label{eq:L-contractive-in-entropy}
    \H(\L_+ f |\M)
    \leq
    (1-\gamma_b) \H(f|\M),
  \end{equation}
  where $\gamma_b$ is defined by \eqref{eq:gammab}.
\end{lemme}

\begin{proof}
  We have, using Lemma \ref{lem:MT} with $g=\M$
  \begin{multline*}
    \H(\L_+ f |\M)
    = H(\L_+ f )
    - \ird \L_+ f  \log \M \d v  \leq
    (1-\gamma_b) H(f) + \gamma_b H(\M)
    - \ird \L_+ f  \log \M \d v.\end{multline*}
  Since $B(q,\xi)$ is a \emph{normalized} Maxwellian collision kernel, we have that $\L f=\L_+(f)-f$ so that
  $$\H(\L_+ f |\M)  \leq
  (1-\gamma_b) \H(f|\M)
  -\gamma_b \ird (f-\M) \log\M \d v
  - \ird \L f  \log \M\,\d v.$$
  Thus, \eqref{eq:L-contractive-in-entropy} reduces to showing that
  \begin{equation*}
    - \gamma_b \ird (f-\M) \log \M \d v
    \leq  \ird \L f \log \M \d v,
  \end{equation*}
  or, in other words, that
  \begin{equation}
    \label{eq:energy-contraction}
    \ird \L f  |v-u_0|^2 \d v
    \leq - \gamma_b \ird (f-\M) |v-u_0|^2 \d v,
  \end{equation}
  where we have used that $\ds\ird \L f  \d v = \ird (f-\M) \d v = 0$. Actually,
  \eqref{eq:energy-contraction} holds with equality, which can be
  checked by an explicit calculation which we give in the following
  Lemma \ref{lem:energy-contraction-lemma} for the convenience of the
  reader.
\end{proof}

The estimate we need in \eqref{eq:energy-contraction} can be obtained
from the fact that the evolution of the temperature in equation
\eqref{eq:linear-Boltzmann} is explicit in the Maxwellian case (which
was already known; see for example \citet{ST04}). We give here a short
proof for completeness:
\begin{lemme}
  \label{lem:energy-contraction-lemma}
  Let $B(q,\xi)=b(\xi)$ be a normalized Maxwellian collision kernel and $f \in
  L^1(\R^d; (1+|v|^2) \d v)$. Then
  \begin{equation}
    \label{eq:energy-contraction-lemma}
    \ird \L f (v) |v-u_0|^2 \d v
    = - \gamma_b \ird (f(v)-\M(v)) |v-u_0|^2 \d v.
  \end{equation}
\end{lemme}

\begin{proof}
  By using the weak form \eqref{weak} of $\L$ and the fact that $\L(\M)
  = 0$, and writing $h := f - \M$ and $\xi := (q \cdot n) / |q|$,
  \begin{equation}
    \label{eq:1}
    \ird \L h (v) |v-u_0|^2 \d v
    =
    \ird\int_{\S}\ird
    h(v)\M(v_*) b(\xi) (|v'-u_0|^2 - |v-u_0|^2 ) \d v_* \d n \d v.
  \end{equation}
  Notice that
  \begin{multline*}
    |v'-u_0|^2 - |v-u_0|^2
    =
    -\xi^2 |v-u_0|^2 - |v_*-u_0|^2 \xi^2 + 2 (v-u_0) \cdot(v_*-u_0) \xi^2
    \\
    - 2 ((v_*-u_0)\cdot n) ((v-u_0) \cdot n)
    + 2 ((v_*-u_0)\cdot n)^2.
  \end{multline*}
  When substituted inside \eqref{eq:1}, several of these terms vanish
  after integration due to either $\ds	\int h \d v = 0$ or the symmetry of $\M$
  about $u_0$. Hence we obtain, using also the normalization of $\M$, that
  \begin{equation*}
    \ird \L h(v) |v-u_0|^2 \d v
    =
    -\ird
    |v-u_0|^2 h(v) \int_{\S} b(\xi) \xi^2 \d n \d v =
    -\gamma_b \ird |v-u_0|^2 h(v) \d v
  \end{equation*}
which is the desired result.
\end{proof}

\begin{nb}
  Notice that \eqref{eq:energy-contraction} can be rewritten as
  \begin{equation*}
    \ird \Q_+(f,\M) |v-u_0|^2 \d v
    \leq (1- \gamma_b) \ird f |v-u_0|^2 \d v
    + \gamma_b \ird \M |v-u_0|^2 \d v\,,
  \end{equation*}
  which strongly resembles \eqref{eq:MT} for the temperature
  functional instead of the relative entropy. Equation
  \eqref{eq:energy-contraction-lemma}, written as
  \begin{equation*}
    \ddt \ird |v-u_0|^2 (f-\M) \d v
    =
    -\gamma_b \ird |v-u_0|^2 (f-\M) \d v
  \end{equation*}
  is also analogous to
  \eqref{entropydiss}, which can be written as
  \begin{equation*}
    \ddt \H(f|\M) \leq -\gamma_b \H(f|\M)
  \end{equation*}
  for a Maxwellian kernel.
\end{nb}

\medskip
We are finally able to complete the proof of Theorem \ref{theo:max}:
\begin{proof}[Proof of Theorem \ref{theo:max}]
  Since in the Maxwellian case we have $\L f =\L_+ f -f$, using the
  expression of $D_\mathrm{max}$ in \eqref{eq:dtH-linear-Boltz} one gets that
  \begin{multline*}
    \D_\mathrm{max}(f)=\IR f
    \log\left(\dfrac{f}{\M}\right)\d\v
    -\IR \L_+(f)\log\left(\dfrac{f}{\M}\right)\d\v
    \\
    =
    \H(f|\M)
    -\H(\L_+ f  | \M)
    +\H(\L_+ f  | f)
    \geq
    \H(f|\M)
    -\H(\L_+ f  | \M),
  \end{multline*}
  since $\ds\ird \L_+ f  \d v= 1 = \ds\ird f \d v$ and
  $$\H(g|f)=\ird g \log \frac{f}{g} \d v \geq
  0$$ whenever $f$ and $g$ share the same mass. Finally, using Lemma
  \ref{lem:max} to estimate $\H(\L_+ f  | \M)$ gives
  \begin{equation*}
    \D(f)
    \geq
    \gamma_b \H(f|\M)
  \end{equation*}
  which is the desired result.
\end{proof}

\section{Inequalities for non-Maxwellian collision kernels}
\label{sec:non-maxwell}

\subsection{Comparison of dissipations for general kernels}
\label{sec:comparison}

As explained in the Introduction, the rest of entropy dissipation
inequalities which we derive are based on Theorem \ref{theo:max},
valid for Maxwellian collision kernels. We then obtain similar
inequalities by comparing the dissipation for a given kernel $B$ with
a Maxwellian dissipation. This strategy was already used in
\citet*{LoMo} in order to estimate the spectral gap for the operator
$\L$ and comes from \cite{MoBa} where it was used to estimate the
spectral gap of the linearized operator $\Q(f,\M) + \Q(\M,f)$. For the
linear Boltzmann operator $\L$, we give here an improved version which
enables us, for example, to estimate the entropy dissipation
functional for the physical case of hard-spheres interactions (we will
also use this comparison principle for grazing collisions kernels in
Section \ref{sec:diffusive}).

Since the linear equation \eqref{eq:linear-Boltzmann} is, as remarked
before, the Kolmogorov forward equation of a Markov process with
equilibrium $\M$, it is well known (see for example \citet{Chafai04})
that all functionals of the form
\begin{equation}
  \label{eq:H-Phi}
  \H_\Phi(f|\M) = \ird M(v) \Phi\left( \frac {f(v)} {\M(v)} \right) \d v,
\end{equation}
for $\Phi:[0,\infty) \to [0, \infty)$ convex, are decreasing along
solutions to \eqref{eq:linear-Boltzmann}. In fact, using the detailed
balance property \eqref{eq:detailed-balance} one sees formally that
\begin{equation*}
  \ddt \H_\Phi(f(t)|\M)
  = \ird \L f(t,v) \, \Phi'\left( \frac {f(t,v)} {\M(v)} \right) \d v
  = -\D_\Phi(f(t)),
\end{equation*}
for any solution $f(t,v)$ to \eqref{eq:linear-Boltzmann}, where
\begin{gather}\label{eq:D_Phi}
  \D_\Phi(f) :=
  \frac{1}{2} \ird \ird \int_{\S}
  B(|q|, \xi)\, \M(v)\M(v_*)
  \Psi\left(
    \frac{f(v)}{\M(v)}, \frac{f(v')}{\M(v')}
    \right)
  \d n \d v \d v_*
  \\
  \Psi(x,y) := (x-y) (\Phi'(x) - \Phi'(y)) \geq 0,
  \qquad x,y \in [0,+\infty).
\end{gather}
Alternatively, we can write
\begin{equation}
  \label{eq:D_Phi-kernel-form}
  \D_\Phi(f) := \frac{1}{2} \ird \ird
  \M(v) k_B(v,v')
  \Psi\left( \frac{f(v)}{\M(v)}, \frac{f(v')}{\M(v')} \right)
  \d v \d v',
\end{equation}
with $k_B$ the kernel of the linear operator $\L$ (see
\eqref{eq:linear-Boltzmann-kernel-form}). (Note, however, that
$\H_\Phi(f|\M)$ is decreasing along solutions also for equations
without detailed balance, though the expression of the dissipation
is different in that case.)

We call $\H_\Phi(f|\M)$ the \emph{relative $\Phi$-entropy} of $f$ with
respect to $\M$. Particular examples of it are given by
$\Phi(x) = x \log x - x + 1$, which gives the usual relative entropy
\eqref{eq:relative-entropy} when $f$ has the same mass as $\M$; and
$\Phi(x) = (x-1)^2$, which gives the distance of $f$ to the
equilibrium $\M$ in the $L^2$ norm with weight $\M^{-1}$.  Since $\L$
depends on the collision kernel $B$, it will be sometimes convenient to
rather write $\D_{\Phi}^B(f)$ to emphasize the collision kernel
$B$. Since our arguments apply to general relative $\Phi$-entropies
with no modification, we state our results for them as well.

\begin{propo}[Comparison of dissipations]
\label{prop:D-comparison-general}
Let $B$, $\tilde{B}$ be two collision kernels defined by
  \begin{equation}
    \label{eq:B-to-compare}
    B(|q|,\xi) = \beta(|q|) b(\xi),
    \qquad
    \tilde{B}(|q|,\xi) = b(\xi)
  \end{equation}
  where $\beta:[0,\infty) \to [0,\infty)$ is a nondecreasing mapping
  and $b(\cdot)$ satisfies the normalization condition
  \eqref{eq:normalized}. Call $\M_0$ the normalized Maxwellian with
  mean velocity $0$ and temperature $\theta > 0$ (that is, $\M_0(v) =
  \M(v+u_0)$).  Assume that there exists $\varrho_0 >0$ such that
  \begin{equation}
    \label{convol}
    \tilde{C}_\theta := \underset{s \in [0,\varrho_0]}{\inf_{\bar{v} \in \R^{d-1}}}
    \frac{\displaystyle \int_{\R^{d-1}}
      \beta \Big(\big(|\bar{v}-\bar{v}_*|^2 + s^2\big)^{1/2}\Big)
      b\left(  \frac{s}{\left(|\bar{v}-\bar{v}_*|^2 + s^2  \right)^{1/2}}
      \right)
      \, \M_0 (\bar{v}_*) \d \bar{v}_*}
    {\displaystyle \int_{\R^{d-1}}
      b\left(  \frac{s} {\left(|\bar{v}-\bar{v}_*|^2 + s^2\right)^{1/2}}
      \right)
      \, \M_0 (\bar{v}_*) \d \bar{v}_*} > 0.
  \end{equation}
  (Where, for $w \in \R^{d-1}$, $\M_0(w)$ is understood as
  $\M_0(w,0)$.) Then
  \begin{equation}
    \label{eq:diss-comparison}
    \D_\Phi^B(f) \geq C_\theta \D_\Phi^{\tilde{B}}(f)
  \end{equation}
  for any probability distribution $f \in L^1(\R^d)$, with $C_\theta
  := \min\{\beta(\varrho_0), \tilde{C}_\theta\}$.
\end{propo}

\begin{nb}
  At first sight, the comparison of convolution integrals
  \eqref{convol} may seem difficult to check. However, we shall see
  further on that it holds true for hard potential interactions (see
  Prop. \ref{prp:D-comparison-hardpotentials}) and for the kernels
  used in the grazing collision limit (see Proposition
  \ref{prp:D-comparison}).
\end{nb}

In order to give the proof of Proposition
\ref{prop:D-comparison-general} we follow the ideas in
\cite[Proposition 3.3]{LoMo}, but we rephrase the argument in a
simplified way. Particularly, we show that Proposition
\ref{prop:D-comparison-general} can actually be deduced from a
comparison of the kernels $k_B$, $k_{\tilde{B}}$ of $\L$ corresponding
to $B$ and $\tilde{B}$.

Notice that the kernel $k_B$ of $\L$ (see expression
\eqref{eq:linear-Boltzmann-kernel-form}) can be calculated by the use
of Carleman's representation (originally described by
\citet{Carleman1957}; see also \citet[section 1.4.6]{Villani02}):
\begin{multline*}
\L_+ f (v)=  \Q_+(f,\M)(v)
  =
  \int_{\R^d} \int_{\S} B(|q|, \xi)
  f(v')\M(v'_*) \d v_* \d n
  \\
  =
  2 \int_{\R^d}  \frac{f(v')}{|v-v'|^{d-1}}
  \int_{E_{v,v'}} B(|q|,\xi) \M(v'_*) \d v'_* \d v',
\end{multline*}
where $E_{v,v'}$ is the hyperplane $\{v'_* \in \R^d \mid (v'_* - v)
\cdot (v'-v)=0 \}$, and it is understood that the $\d v'_*$ integral
above is with respect to the $(d-1)$-dimensional Lebesgue measure on this
hyperplane. Since $|q|$ and $\xi$ must now be written in terms of
$v$, $v'_*$ and $v'$, note that
\begin{equation}
  \label{eq:Carleman-change-q-xi}
  |q| = |2v - v' - v'_*|,
  \qquad
  |q \cdot n| = |v-v'|,
  \qquad
  \xi = \frac{|v-v'|}{|2v - v' - v'_*|}.
\end{equation}
Hence we have, for any $B=B(|q|,\xi)$:
\begin{equation}
  \label{eq:kB}
  k_B(v',v) = \frac{1}{|v-v'|^{d-1}} \int_{E_{v,v'}} B(|q|,\xi) \M(v'_*) \d v'_*,
  \qquad v',v \in \R^d.
\end{equation}
We now prove the following which clearly implies Proposition
\ref{prop:D-comparison-general} by virtue of
\eqref{eq:D_Phi-kernel-form}:
\begin{propo}[Comparison of kernels]
  \label{lem:kernel-comparison}
  Assume that the collision kernels $B$ and $\tilde{B}$ satisfy
  \eqref{convol}. Then, for the same constant $C_\theta$ as in
  Proposition \ref{prop:D-comparison-general},
  \begin{equation}\label{eq:kBkmax}
    k_B(v',v) \geq C_\theta\, k_{\tilde{B}} (v',v)
    \quad \text{ for all } v,v' \in \R^d.
  \end{equation}
\end{propo}

\begin{proof}
  By translation invariance of \eqref{eq:linear-Boltzmann} (i.e.,
  $\Q(f,\M)(v+u) = \Q(f(\cdot+u), \M(\cdot+u))(v)$) one sees that it
  is enough to show the result when the mean velocity $\M$, namely
  $u_0$, is equal to $0$ (in fact, the kernels $k_B$ corresponding to
  different mean velocities are just translations of one another).
  Hence we assume $u_0 = 0$ throughout the proof, so $\M = \M_0$ (this
  will make calculations easier).

  Using \eqref{eq:Carleman-change-q-xi} and \eqref{eq:kB},
  \eqref{eq:kBkmax} is equivalent to
  \begin{multline}
    \label{eq:5}
    \int_{E_{v,v'}} \beta( |2v - v' - v'_*|)
    \,b\left(\frac{|v-v'|}{|2v - v' - v'_*|}\right)M_0(v'_*) \d v'_*
    \geq
    C_\theta \int_{E_{v,v'}} \,b\left(\frac{|v-v'|}{|2v - v' -
        v'_*|}\right)M_0(v'_*) \d v'_*
    \\ \text{ for any } v',v \in \R^d.
  \end{multline}
  Take $n$ to be the unit vector along the direction of
  $v-v'$.  We
  now write $$v'_* = r n + \bar{v}'_*$$ for (uniquely determined)
  $r \in \R$ and $\bar{v}'_*$ orthogonal to $n$. Write also $$v =
  r n+ \bar{v}$$ for some $v$ orthogonal to $n$ (note that
  $r$ must have the same value as before, since $v'_*-v$ is orthogonal
  to $n$ in $E_{v,v'}$) and $$v' = (r+s)n + \bar{v}$$ for
  some $s \in \R$ (and the same $\bar{v}$ as before, since $v-v'$ is
  parallel to $n$.)  With this and the expressions in
  \eqref{eq:Carleman-change-q-xi} we have
  \begin{gather}
    \label{eq:q-change}
    |2v-v'-v'_*|^2
    = |v - v'|^2 + |v - v'_*|^2
    = |\bar{v} - \bar{v}'_*|^2 + s^2, \qquad
    \frac{|v-v'|}{|2v-v'-v'_*|}= \frac{s}{\sqrt{ |\bar{v} - \bar{v}'_*|^2 + s^2 }}.
  \end{gather}
  Changing variables to $\bar{v}'_*$, we obtain that \eqref{eq:5} reads
  \begin{equation*}
    \int_{n^\perp}  \beta \left( \sqrt{|\bar{v} - \bar{v}'_*|^2 + s^2} \right)
    \,b\left(\frac{s}{\sqrt{ |\bar{v} - \bar{v}'_*|^2 + s^2 }}\right)
    M_0(\bar{v}'_*) \d \bar{v}'_*
    \geq
    C_\theta \int_{n^\perp}
    b\left(\frac{s}{\sqrt{ |\bar{v} - \bar{v}'_*|^2 + s^2 }}\right)
    M_0(\bar{v}'_*) \d \bar{v}'_*.
  \end{equation*}
Notice that we have used here that $\M(v'_*) =
(2\pi\theta)^{-d/2}\M(\bar{v}'_*) \M(r n)$, with $\M(r n)$
independent of the integration variable and hence cancelling from both
sides of the inequality.

  By rotational symmetry we may also take $v-v'$ parallel to
  $(0,\cdots,0,1) \in \R^{d}$, so that $n^\perp = \R^{d-1}$,
  identified as the set of points in $\R^d$ with zero last coordinate
  (so the variables with a bar just represent the first $d-1$
  coordinates of the variables without a bar). Then, \eqref{eq:5} is equivalent to
  \begin{multline}
    \label{eq:7}
    \int_{\R^{d-1}}  \beta\left( \sqrt{|\bar{v} - \bar{v}'_*|^2 +
        s^2}\right)
    \,b\left(\frac{s}{\sqrt{ |\bar{v} - \bar{v}'_*|^2 + s^2 }}\right)
    \M_0(\bar{v}'_*) \d \bar{v}'_*
    \geq
    C_\theta \int_{\R^{d-1}} b
    \left(
      \frac{s}{\sqrt{ |\bar{v} - \bar{v}'_*|^2 + s^2 }}
    \right)
    \M_0(\bar{v}'_*) \d \bar{v}'_*
    \\
    \text{ for any $\bar{v} \in \R^{d-1}$ and $s \geq 0$.}
  \end{multline}
  Now, given $\varrho_0 > 0$, since $\beta(\cdot)$ is nondecreasing,
  it is clear that, for any $\bar{v} \in \R^{d-1}$ and any $s \geq
  \varrho_0$ it holds
  \begin{multline*}
   \int_{\R^{d-1}} \beta\left( \sqrt{|\bar{v} - \bar{v}'_*|^2 + s^2}\right)\,b\left(\frac{s}{\sqrt{ |\bar{v} - \bar{v}'_*|^2 + s^2 }}\right) \M_0(\bar{v}'_*) \d \bar{v}'_*\\
    \geq
    \beta(\varrho_0) \int_{\R^{d-1}} b\left(\frac{s}{\sqrt{ |\bar{v} - \bar{v}'_*|^2 + s^2 }}\right) \M_0(\bar{v}'_*) \d \bar{v}'_*.
  \end{multline*}
  It is then enough to show that \eqref{eq:7} holds for some constant
  $\tilde{C}_\theta$, uniformly for $\bar{v} \in \R^{d-1}$ and $s \in
  [0,\varrho_0)$. This is exactly assumption \eqref{convol}. This
  achieves the proof and, in particular, shows that \eqref{eq:kBkmax}
  holds with $C_\theta = \min(\beta(\varrho_0), \tilde{C}_\theta)$.
\end{proof}

By using Theorem \ref{theo:max}, Proposition
\ref{prop:D-comparison-general} directly implies the following
inequality for non-Maxwellian collision kernels:
\begin{theo}
  \label{thm:main-general}
  Assume that the collision kernel $B$ is given by
  \eqref{eq:B-to-compare} where $b(\xi)$ is a normalized Maxwellian
  collision kernel and $\beta(\cdot)$ satisfies \eqref{convol}. Then
  for all nonnegative probability distributions $f$ we have
  \begin{equation}
    \label{eq:EPI-non-Maxwell}
    \D_B(f) \geq C_\theta \gamma_b \, \H(f|\M),
  \end{equation}
  where $C_\theta > 0$ is the constant in Proposition
  \ref{prop:D-comparison-general} and $\gamma_b$ was defined in
  \eqref{eq:gammab}.
\end{theo}

Remark \ref{rem:non-cutoff} also applies here: the above result is
valid for any collision kernels which can be bounded below by a
collision kernel satisfying the hypotheses of the theorem, and in
particular it applies to non-cut-off hard collision kernels.

\subsection{Application to hard-potential interactions}

We show here how the above Proposition applies to the fundamental
model of hard-potential interactions (including the hard-spheres case)
for which
\begin{equation}
  \label{eq:Bhard}
  B(|q|,\xi)=c_d\,|q|^\gamma \,\xi^{d-2}
\end{equation}
where $c_d > 0$ is a normalization constant given by
$$c_d:=\left(|\S|\int_0^1
  \xi^{d-2}  \,\left(1-\xi^2\right)^{\frac{d-3}{2}}\d\xi\right)^{-1}.
$$
Introducing then $\tilde{B}(|q|,\xi)=b(\xi)=c_d \xi^{d-2}$, we see
that $\tilde{B}$ is a normalized Maxwellian collision kernel.
As a consequence of Proposition \ref{prop:D-comparison-general} we
obtain the following, which completes the proof of Theorem \ref{main}:

 \begin{propo}
  \label{prp:D-comparison-hardpotentials}
  Let $B$ be a hard-potential collision kernel of the form
  \eqref{eq:Bhard} with $\gamma \geq 0$, in dimension
  $d \geq 2$. There exists some explicit $C > 0$ such that
  \begin{equation}
    \label{eq:HS-comparison}
    \D_\Phi^{B}(f) \geq C \theta^{\gamma/2} \D_\Phi^{\tilde{B}}(f).
  \end{equation}
\end{propo}

\begin{proof}
  The proof consists simply in checking that Assumption \eqref{convol}
  is met by the kernels $\beta(|q|)=|q|^{\gamma}$ and $b(\xi)=c_d
  \xi^{d-2}.$ Actually, the dependence on $\theta$ is easily obtained:
  call, for $\mu > 0$,
  \begin{equation*}
    M_\mu(v) := \mu^d M(\mu v),
    \quad
    f_\mu(v) := \mu^d f(\mu v).
  \end{equation*}
  Note that the temperature of $M_\mu$ is $\mu^{-2}$ times that of
  $M$. Then we have the scaling
  \begin{equation*}
    \D_{\Phi,M_\mu}^{B}(f_\mu) = \mu^{-\gamma} \D_{\Phi,M}^{B}(f),
    \quad
    \D_{\Phi,M_\mu}^{\tilde{B}}(f_\mu) = \D_{\Phi,M}^{B}(f),
  \end{equation*}
  where we have denoted the dependence on $M$ as an additional
  subscript. One sees then that it is enough to show
  \eqref{eq:HS-comparison} when the temperature $\theta$ of $M$ is
  equal to 1, so we assume this in the rest of the proof.

  To prove \eqref{convol} it suffices clearly to show that there
  exists $C >0$ such that
  \begin{equation}
    \label{eq:6}
    \int_{\R^{d-1}} \M_0(\bar{v}_*)
    (\sqrt{|\bar{v}-\bar{v}_*|^2+s^2})^{2-d+\gamma}
    \d \bar{v}_*
    \geq C\int_{\R^{d-1}} \M_0(\bar{v}_*)
    (\sqrt{|\bar{v}-\bar{v}_*|^2+s^2})^{2-d}
    \d \bar{v}_*
  \end{equation}
  for any $\bar{v} \in \R^{d-1}$ and any $0 < s \leq 1$. Choose
  $\delta > 0$. In the region where $|\bar{v}-\bar{v}_*| \geq \delta$
  we have $\sqrt{|\bar{v}-\bar{v}_*|^2+s^2} \geq \delta$ and
  hence
  \begin{equation*}
    \int_{|\bar{v}-\bar{v}_*| \geq \delta} \M_0(\bar{v}_*)
    (\sqrt{|\bar{v}-\bar{v}_*|^2+s^2})^{2-d+\gamma}
    \d \bar{v}_*
    \geq
    \delta^\gamma
    \int_{|\bar{v}-\bar{v}_*| \geq \delta} \M_0(\bar{v}_*)
    (\sqrt{|\bar{v}-\bar{v}_*|^2+s^2})^{2-d}
    \d \bar{v}_*.
  \end{equation*}
  So it is enough to show that
  \begin{equation}
    \label{eq:8}
    \int_{|\bar{v}-\bar{v}_*| < \delta} \M_0(\bar{v}_*)
    (\sqrt{|\bar{v}-\bar{v}_*|^2+s^2})^{2-d}
    \d \bar{v}_*
    \leq
    K
    \int_{\R^{d-1}} \M_0(\bar{v}_*)
    (\sqrt{|\bar{v}-\bar{v}_*|^2+s^2})^{2-d+\gamma}
    \d \bar{v}_*
  \end{equation}
  for some $K > 0$, all $\bar{v} \in \R^d$ and all $0 < s \leq
  1$, which would imply \eqref{eq:6} to hold  with $C = \min\{\delta^\gamma,
  1/K\}$. Let us bound the left-hand-side of \eqref{eq:8} first. On
  the integration region we have $|\bar{v}_*| \geq
  (|\bar{v}|-\delta)_+ =: \max\{|\bar{v}|-\delta, 0\}$. Hence
  \begin{equation*}
    \M_0(\bar{v}_*) \leq
    K_1 \exp(-|\bar{v}_*|^2/2)
    \leq K_1
    \exp\left( -\frac{(|\bar{v}|-\delta)_+^2}{2} \right)
  \end{equation*}
  for some $K_1 > 0$. Using this, the left hand side of \eqref{eq:8}
  is bounded above by
  \begin{equation}\begin{split}
    \label{eq:9}
    \int_{|\bar{v}-\bar{v}_*| < \delta} \M_0(\bar{v}_*)
    (\sqrt{|\bar{v}-\bar{v}_*|^2+s^2})^{2-d}
    \d \bar{v}_*
    &\leq
    \int_{|\bar{v}-\bar{v}_*| < \delta} \M_0(\bar{v}_*)
    |\bar{v}-\bar{v}_*|^{2-d}
    \d \bar{v}_*
    \\
    &\leq K_2 \exp \left( -\frac{(|\bar{v}|-\delta)_+^2}{2} \right),
  \end{split}\end{equation}
  for some $K_2 > 0$. On the other hand, using that for $|\bar{v}_*| <
  1$ we have
  \begin{equation*}
    \sqrt{|\bar{v}-\bar{v}_*|^2+1} \leq
    \sqrt{(|\bar{v}| + 1)^2+1}
    \leq
    K_4 (|\bar{v}| + 1)
  \end{equation*}
  for some $K_4 > 1$, we see that the right hand side of \eqref{eq:8}
  is bounded below by
  \begin{multline}
    \label{eq:10}
    \int_{\R^{d-1}} \M_0(\bar{v}_*)
    (\sqrt{|\bar{v}-\bar{v}_*|^2+s^2})^{2-d+\gamma}
    \d \bar{v}_*
    \geq
    \int_{|\bar{v}_*|<1} \M_0(\bar{v}_*)
    (\sqrt{|\bar{v}-\bar{v}_*|^2+1})^{2-d+\gamma}
    \d \bar{v}_*
    \\
    \geq
    K_5 \left(1+|\bar{v}|\right)^{2-d+\gamma}
  \end{multline}
  when $d \geq 2+\gamma$, or simply by
  \begin{equation}
    \label{eq:11}
    \int_{\R^{d-1}} \M_0(\bar{v}_*)
    (\sqrt{|\bar{v}-\bar{v}_*|^2+s^2})^{2-d+\gamma}
    \d \bar{v}_*
    \geq
    \int_{\R^{d-1}} \M_0(\bar{v}_*) |\bar{v}-\bar{v}_*|^{2-d+\gamma}
    \d \bar{v}_*
    \geq
    K_6,
  \end{equation}
  for some $K_6 > 0$, when $d < 2 + \gamma$. The bounds \eqref{eq:9},
  \eqref{eq:10} and \eqref{eq:11} clearly show \eqref{eq:8}, finishing
  the proof of the lemma.
\end{proof}

\begin{exa}[Hard-spheres case in dimension $3$]
  \label{exa:hard-spheres}
  Let us estimate the constant $C$ above in general dimension $d \geq
  2$ whenever $\gamma = d-2$, which happens to be slightly easier and
  covers in particular the physically relevant case of hard-spheres in
  dimension $d=3$ for which $\gamma=1$. Let us then assume that $d
  \geq 2$ and let
  \begin{equation*}
    \beta(|q|)=|q|^{d-2}
    \qquad \text{ and }
    \qquad  b(\xi)=c_{d}\,\xi^{d-2}.
  \end{equation*}
 Then,  for any $s \geq 0$, with the notations of Proposition \ref{prop:D-comparison-general},
\begin{equation*}\beta \Big(\big(|\bar{v}-\bar{v}_*|^2 + s^2\big)^{1/2}\Big)
      b\left(  \frac{s}{\left(|\bar{v}-\bar{v}_*|^2 + s^2  \right)^{1/2}}
      \right)=c_{d}\,s^{d-2}\end{equation*}
so that, to check \eqref{convol}, it is enough to
  show the inequality
  \begin{equation}
    \label{eq:12}
    \int_{\R^{d-1}} M_{0}(\bar{v}_*)\,\d \bar{v}_*
    \geq  \tilde{C}_\theta \int_{\R^{d-1}} \frac{M_{0}(\bar{v}_*)}
    {|\bar{v}-\bar{v}_*|^{d-2}}\,\d \bar{v}_*
  \end{equation}
  for any $\bar{v} \in \R^{d-1}$. Now, the left hand side is a given
  number, namely
  \begin{equation*}
    \int_{\R^{d-1}} M_{0}(\bar{v}_*)\,\d \bar{v}_*
    =\|M_{0}\|_{L^{1}(\R^{d-1})}=\frac{1}{\sqrt{2\pi\theta}}
  \end{equation*}
  while the right hand side is bounded for $\bar{v} \in \R^{d-1}$: one
  can write, for any $\bar{v} \in \R^{d-1}$ and any $r > 0$,
  \begin{multline*}
    \int_{\R^{d-1}} \dfrac{\M_0(\bar{v}_*)}
    {|\bar{v}-\bar{v}_*|^{d-2}}\,\d \bar{v}_*
    \leq
    \|\M_0\|_{L^\infty(\R^{d-1})}\int_{\{|\bar{v}-\bar{v}_*| < r\}}
    \dfrac{\d\bar{v}_*}{|\bar{v}-\bar{v}_*|^{d-2}} +
    r^{-(d-2)}\|\M_0\|_{L^1(\R^{d-2})}
    \\
    =
    r\,\|\M_0\|_{L^\infty(\R^{d-1})}\,|\mathbb{S}^{d-2}|
      + r^{-(d-2)}\|\M_0\|_{L^1(\R^{d-1})}
  \end{multline*}
  and, optimizing the parameter $r > 0$, one finds that the constant
  $\tilde{C}_\theta > 0$ in \eqref{eq:12} can be chosen as
  \begin{equation*}
    C_{0} := \frac{1}{d-1}
    \left( \dfrac{(d-2) \|\M_0\|_{L^1(\R^{d-1})}}
      {\|M_{0}\|_{L^{\infty}(\R^{d-1})}\,|\mathbb{S}^{d-2}|}
    \right)^{\frac{d-2}{d-1}}
  \end{equation*}
  In particular, in dimension $d=3$, one can choose $\tilde{C}_\theta
  = \frac{\sqrt{\theta}}{2}$. For the special case of the
  \textit{Shannon-Boltzmann relative entropy}, i.e. for $\Phi(x)=x\log
  x-x+1$, one simply denotes by $\D_\mathrm{max}$ the dissipation
  associated to $\tilde{B}(\xi)=c_d \xi$ and, bearing in mind that $c_d$ is a normalization constant for collision kernel, deduces from Theorem
  \ref{theo:max} that:
  $$\D_\mathrm{max}(f) \geq \frac{1}{2} \H(f|\M).$$
  Therefore, if $\D_\mathrm{hs}$ denotes the entropy dissipation
  associated to hard-spheres interactions in dimension $d=3$ we
  immediately deduce from Proposition
  \ref{prp:D-comparison-hardpotentials} that
  $$\D_\mathrm{hs}(f) \geq
  \dfrac{\sqrt{\theta}}{4} \H(f|\M).$$
\end{exa}

\section{Some applications} \label{sec:speed1}

\subsection{Speed of convergence to equilibrium for the linear
  Boltzmann equation}
\label{sec:speed}

Once we have Theorem \ref{main} and the relation
\eqref{eq:dtH-linear-Boltz} it is straightforward to deduce the
following result:
\begin{theo} Let $B=B(|q|,\xi)$ denote a hard-potential collision
  kernel given by \eqref{eq:hard-potential-kernel} or any normalized
  Maxwellian collision kernel $B=b(\xi)$ satisfying
  \eqref{eq:normalized}. Let $f_0 \in L^1(\R^d,(1+|v|^2)\d v)$ be a
  given probability density with finite entropy and let
  $f(t)=f(t,\cdot)$ be the associated solution to the linear Boltzmann
  equation \eqref{eq:linear-Boltzmann}. Then
  \begin{equation}
    \label{eq:H-decay}
    \H(f(t)|\M) \leq \exp(-\lambda\,t) \H(f_0|\M)
    \quad
    \text{ for \quad $t \geq 0$},
  \end{equation}
  with $\lambda$ given in Theorem \ref{main}. In particular, the
  Csisz\'ar-Kullback-Pinsker inequality yields
  $$\|f(t)-\M\|_{L^1(\R^d)} \leq \sqrt{2}\exp\left(-\frac{\lambda}{2}\,t\right) \H(f_0|\M) \qquad \forall t \geq 0.$$
\end{theo}
Again, this result applies also to any collision kernels that can be
bounded below by a kernel satisfying the assumptions; see Remark
\ref{rem:non-cutoff}.

For completeness we gather here some facts on the well-posedness of
equation \eqref{eq:linear-Boltzmann} and the rigorous derivation of
the entropy relation \eqref{eq:dtH-linear-Boltz}. Assume for the rest
of this paragraph that
\begin{equation*}
  B(|q|,\xi) = |q|^\gamma b(\xi)
\end{equation*}
for some $0 \leq \gamma \leq 2$ and some collision kernel $b(\cdot)$
satisfying \eqref{eq:normalized}. First, we notice that the operator
$\L(f)$ is well defined for $f \in L^1(\R^d; (1+|v|^\gamma) \d
v)$. When considered as an operator on $L^2(\R^d;\M^{-1})$ (a smaller
space than $L^1(\R^d; (1+|v|^\gamma)$) then $\L$ is a self-adjoint
operator with domain $L^2(\R^d; \M(v)^{-1} |v|^\gamma \d v)$ (see
\cite{Carleman1957}). Regarding the evolution equation
\eqref{eq:linear-Boltzmann}, $\L$ (with its natural domain) generates
a $C_0$-semigroup in several spaces; for example, in
$L^2(\R^d;\M^{-1})$ and in $L^1(\R^d)$. By considering a suitable
regularization $\Phi_\epsilon:[0,+\infty)$ of the function $\Phi(x) :=
x \log x-x+1$, with $\Phi_\epsilon$ differentiable on $[0,+\infty)$,
one directly sees that, for any $\epsilon > 0$, it holds
\begin{equation*}
  \ddt \H_{\Phi_\epsilon}(f(t)) = -\D_{\Phi_\epsilon}(f(t))
\end{equation*}
for any  solution $f(t)$ to \eqref{eq:linear-Boltzmann} (in the semigroup
sense) with initial condition in the domain of $\L$. One can then pass
to the limit in $\epsilon \to 0$ in order to show that
\eqref{eq:dtH-linear-Boltz} holds rigorously for an initial condition
$f$ with finite energy and entropy.

\subsection{Trend to equilibrium for the nonlinear Boltzmann equation
  with particle bath}

We consider now the nonlinear (elastic) Boltzmann operator with
particles bath
\begin{equation}\label{nonline}
  \partial_t f(t,v)=\alpha \Q(f,f)(t,v) + \L f , \qquad f(0,v)=f_0(v),
  \qquad t \geq 0, \ v \in \R^d,
\end{equation}
where $\alpha \geq 0$ is a given constant while, as above, $\L  f $
denotes the linear Boltzmann operator and $\Q(f,f)$ is the quadratic
Boltzmann operator. Equation \eqref{nonline} models the evolution of
particles (typically hard-spheres) according to the following rules:
particles are suffering binary collision with themselves and also
interact with the particles of a host medium at thermodynamical
equilibrium.  Notice that \eqref{nonline} has been recently considered
in \cite{BCL11} (for inelastic interactions) and can also be seen as
the spatially homogeneous version of the model recently investigated
in \cite{frohlich}.

In the above, one assumes that $\Q=\Q_{B_1}$ is associated to a
general collision kernel $B_1(|q|,\xi) \geq 0$ (including soft
interactions, see \cite{Villani02}). Moreover, one assumes that $\L$
is associated to a collision kernel $B(|q|,\xi)=\beta(|q|)b(\xi)$
where $b(\cdot)$ satisfies the normalization condition
\eqref{eq:normalized} while $\beta(\cdot)\::\:[0,\infty)\to
[0,\infty)$ satisfies \eqref{convol}. Notice that we do not need here
$B_1$ and $B$ to be equal.

The well-posedness of the Cauchy problem associated to \eqref{nonline}
can be handled with standard methods from spatially homogeneous kinetic
theory and we do not address this question here, referring for
instance to \cite{Villani02} for more details (see also \cite{Bcar}
where a similar equation has been investigated for inelastic
interactions). Moreover, it is also easy to prove that the unique
steady state of the operator $\alpha \Q(f,f)+ \L(f)$ is the host
Maxwellian $\M$, namely (see \cite{BCL11}):
\begin{propo}
  \label{propo:unique-elastic}
  For any $\alpha \geq 0$, the unique nonnegative solution $F \in
  L^1(\R^d; (1+|v|)^2 \d v)$ with unit mass to the stationary problem
  $$\alpha \Q(F,F) + \L(F)=0$$
is given by $F(v)=\M(v).$
\end{propo}
Concerning the long time behavior of solution, we prove in a simple
way exponential trend towards equilibrium:
\begin{theo}
  For any $\alpha \geq 0$, let $f_0 \in L^1(\R^d, (1+|v|)^3\d v)$ be
  such that $\H(f_0|\M) < \infty$ and let $f(t,v)$ be the unique
  associated global solution to \eqref{nonline}. Then, there exists $C
  > 0$ depending only $B$ such that
  \begin{equation}
    \label{HftM}
    \H(f(t)|\M) \leq \exp(-C\,\gamma_b t) \H(f_0|\M) \qquad \forall t
    \geq 0
\end{equation}
where $\gamma_b >0$ is the constant appearing in Theorem
\ref{theo:max} while $C > 0$ is the constant appearing in
Prop. \ref{prop:D-comparison-general}.
\end{theo}

\begin{nb}
  Notice that the long-time behavior of the solution $f(t,v)$ is
  completely driven by $\L$ and not by the quadratic operator $\Q$. In
  particular, the speed of convergence does not depend on $\alpha$ and
  is entirely determined by the collision kernel $B=B(|q|,\xi)$.
\end{nb}

\begin{proof}
  The proof follows from standard arguments. Namely,
  direct computations yield
  $$\dfrac{\d}{\d t} \H(f(t)|\M)
  =\alpha\, \IR \Q(f,f)(t,v) \log
  \left(\frac{f(t,v)}{\M(v)}\right)\d\v- D(f)$$ Now, one sees that,
  since $\Q(f,f)$ conserved mass, momentum \textit{and kinetic
    energy},
  $$\IR \Q(f,f)(t,v)\log \left(\frac{f(t,v)}{\M(v)}\right)\d\v
  =\IR \Q(f,f)(t,v)\log f(t,v)\d v$$
  and, by well-known arguments that can be traced back to Boltzmann
  himself, this last quantity is nonnegative (this is exactly the
  classical Boltzmann's $H$-Theorem; see
  eq.~\eqref{eq:D-nl-Boltzmann}). Thus
  $$\dfrac{\d}{\d t} \H(f(t)|\M) \leq - D(f(t))$$
  and, one deduces from Theorem \ref{thm:main-general} that
  $$\dfrac{\d}{\d t} \H(f(t)|\M)
  \leq - C \gamma_b \H(f(t)|\M) \qquad \forall t \geq 0$$
  which achieves the proof.
\end{proof}

\section{Grazing collisions limit and logarithmic Sobolev inequalities}
\label{sec:diffusive}

In this section we show how general functional inequalities of the
type \eqref{entropydiss} allow to recover, in a suitable limit, a
well-known entropy-entropy dissipation estimate for a certain linear
Fokker-Planck equation. The limit procedure is the so-called grazing
collisions limit for which we assume that the collision kernel $B$ is
concentrated on small angle deviations. Before describing how this
grazing collisions limit allows to recover a well-known logarithmic
Sobolev inequality, we describe in more detail the asymptotic
procedure.

\subsection{The asymptotics of grazing collisions}
\label{sec:grazing}

Whenever collisions concentrate around $|q \cdot
n|/|q| \simeq 0$, it is well documented that $\L$ becomes close (in a
sense to be made precise) to a certain linear Fokker-Planck operator
(associated to a certain diffusion matrix $\mathbf{D}(v)$ that depends
on $\M$). We explain here the general mathematical framework following
the lines of \cite{LoTo}.

We restrict ourselves to dimension $d=3$ for simplicity. For any
$\epsilon \in (0,1]$, we consider
\begin{equation}\label{eq:Bgammaeps}B(|q|,\xi)=|q|^\gamma \,{b}_\epsilon(\xi)\end{equation}
for $\gamma=0$ or $\gamma=1$ and with $b_\epsilon(\cdot)$ given by
\begin{equation*}
b_\epsilon(\xi) = \frac{\xi}{2\pi\epsilon}\, \1_{[0, \epsilon]} (\xi)
\end{equation*}
where we recall that $\xi =|q \cdot n|/|n|\,$. Notice that
$b_\epsilon$ is a \emph{normalized} Maxwellian collision kernel. Let
$\L_\epsilon$ denote the associated linear Boltzmann operator (we do
not distinguish here the two cases $\gamma=0$ --- corresponding to
Maxwellian collision kernel --- and $\gamma=1$ corresponding to
hard-spheres). Given $f_0 \in L^1(\R^3,(1+|v|^2)\d v)$, let
$h=h_\epsilon(t,v)$ denote the unique solution to
\begin{equation}\label{eq:Leps}\partial_t h = \L_\epsilon  h ,
  \qquad h(0,v) = f_0(v)
  \qquad (t \geq 0, v \in \R^3).\end{equation}
Moreover, we introduce the following time scaling
\begin{equation}
\label{scaling}
f_\epsilon (t,v) = h \left(t{\epsilon^{-2}}, v \right) \qquad \forall t \geq 0.
\end{equation}
Using the weak form of the Boltzmann operator provided by \eqref{weak}, for a general test function $\varphi(v)$ one gets that
\begin{equation}
\label{weakBE}
\ddt \ir3 f_\epsilon(t,v)\, \varphi(v) \d v = \frac{1}{\epsilon^2} \ir3 \ir3 f_\epsilon(t,v) {M}(\vb)|v-\vb|^\gamma \d v\d\vb\int_{\S} b_\epsilon(\xi) \big[ \varphi(\v')-\varphi(\v)\big]\d \n.
\end{equation}
Using a Taylor expansion of $\varphi$, one has\footnote{Given two
  vectors $w,v \in \R^{3}$, we use in the sequel the tensor notation
  $w\otimes v$ to denote the matrix with entries $w_{i}v_{j}$
  $i,j=1,2,3$. In particular, the matrix product
  $\mathbb{D}^2 \varphi(v) \cdot \big[ (\v' - \v) \otimes (\v' - \v)
  \big]$
  simply denotes $(v'-v) (\mathbb{D}^2 \varphi) (v'-v)^{\top}$.}
\begin{equation}
\varphi(\v') = \varphi(\v) + \nabla_v \varphi(\v) \cdot (\v' - \v) + \frac12\, \mathbb{D}^2 \varphi(v) \cdot \big[ (\v' - \v) \otimes (\v' - \v) \big] + o(|\v' -\v|^2)
\end{equation}
where $\mathbb{D}^2 \varphi$ is the Hessian matrix of $\varphi$, with
components $\left( \mathbb{D}^2 \varphi(v) \right)_{ij} = \dfrac{\p^2
  \varphi(\v)}{\p v_i \p v_j}$ $(i,j=1,2,3)$;
hence, taking into account (\ref{vprime})
\begin{equation*}
\varphi(\v') = \varphi(\v) -\nabla_v \varphi(\v) \cdot (q \cdot n) n + \frac12\, \mathbb{D}^2 \varphi(\v) \cdot \big[ |q \cdot n|^2 n \otimes n \big] + o(|\v' -\v|^2).
\end{equation*}
Let us evaluate integrals over the angular variable $n$, taking the
direction of the relative velocity $q$ as polar axis ($\hat{e}_3$).
It is easy to check that
\begin{equation*}
\int_{\mathbb{S}^2} b_\epsilon(\xi) (q \cdot n) n \d n = 2 q \int_0^\epsilon \xi^3 \d \xi = \frac{\epsilon^2}{2}\, q\,,
\end{equation*}
(the other components vanishing by parity arguments), while
\begin{equation*}
\begin{array}{c}
\dis \int_{\mathbb{S}^2} b_\epsilon(\xi) |q \cdot n|^2 n \otimes n \d n =
|q|^2 \int_0^\epsilon \xi^3 \Big[ (1- \xi^2) (\hat{e}_1 \otimes \hat{e}_1 + \hat{e}_2 \otimes \hat{e}_2) + 2 \xi^2 \hat{e}_3 \otimes \hat{e}_3 \Big] \d \xi \vspace*{0.2 cm} \\
\dis = |q|^2 \left[ \left( \frac{\epsilon^2}{4} - \frac{\epsilon^4}{6} \right) (\hat{e}_1 \otimes \hat{e}_1 + \hat{e}_2 \otimes \hat{e}_2) + \frac{\epsilon^4}{6} \hat{e}_3 \otimes \hat{e}_3 \right] .
\end{array}
\end{equation*}
By inserting all these results into \eqref{weakBE}  we get
\begin{multline}
\label{FPprelim}
\ddt  \ir3 f_\epsilon(t,v)\, \varphi(v) \d v = \ir3 \ir3 f_\epsilon(t,\v)M(\vb)\Big[ -\, \frac12 (\v - \vb) \cdot \nabla_{\v} \varphi(\v)\\
 + \frac18\, \mathbb{D}^2 \varphi(\v) \cdot \Big( |q|^2 {\bf I} - q \otimes q \Big) \Big] \d  v \d \vb
 + O(\epsilon^2)
\end{multline}
where ${\bf I}$ is the identity matrix. Set
$$
{\bf S}(v,\vb) = |v-\vb|^2 {\bf I} - (v-\vb) \otimes (v-\vb).
$$
By considering the last term in (\ref{FPprelim}) we see that (here and
below we use Einstein's summation convention on repeated indices)
\begin{equation*}\begin{split}
\dis f_\epsilon(t,\v) \mathbb{D}^2 \varphi &\cdot \Big( |v-\vb|^2 {\bf I} - (v-\vb)  \otimes (v-\vb) \Big) = f_\epsilon(t,v) \frac{\p}{\p v_j} \left( \frac{\p \varphi}{\p v_i} \right) \mathbf{S}_{ij}(v,\vb)\\
&= \frac{\p}{\p v_j} \left( f_\epsilon(t,\v) \mathbf{S}_{ij}(v,\vb) \frac{\p \varphi(\v)}{\p v_i} \right)
- \frac{\p \varphi(\v)}{\p v_i} \mathbf{S}_{ij}(v,\vb)\frac{\p f_\epsilon(t,\v)}{\p v_j}   - \frac{\p \varphi(\v)}{\p v_i} \frac{\p \mathbf{S}_{ij}(v,\vb)}{\p v_j} f_\epsilon(t,\v) \\
&= \nabla_{\v} \cdot \Big( f_\epsilon(t,v) {\bf S}(v,\vb) \cdot\nabla_{\v} \varphi(\v)\Big)
- \nabla_{\v} \varphi(\v) \cdot \Big( {\bf S}(v,\vb)\nabla_v f_\epsilon(t,v) - 2 (\v -\vb) f_\epsilon(t,\v) \Big)
\end{split}\end{equation*}
where we used that $ \frac{\p \mathbf{S}_{ij}(v,\vb)}{\p
  v_j}=-(v-\vb)_i\1_{j \neq i}$. Therefore \eqref{FPprelim} may be
cast as
\begin{multline*}
  \label{eqmethod1}
  \dis \ddt \ir3 f_\epsilon(t,v)\, \varphi(v) \d v
  \\
  = - \frac18 \ir3 \ir3 \nabla_v \varphi(\v) \cdot
  \Big[ 2(\v - \vb) f_\epsilon(t,\v)
    + {\bf S}(\v,\vb) \nabla_v f_\epsilon(t,\v)
  \Big]
  \M(\vb)|v-\vb|^\gamma \d v \d \vb +O(\epsilon^2).
\end{multline*}
Noticing  that  $\nabla_{\vb} \cdot \mathbf{S}(v,\vb)=2(v-\vb)$ while $\mathbf{S}(v,\vb)(v-\vb)=0$, we check that, for both $\gamma=0,1$, it holds
$$\nabla_{\vb} \cdot \left(|v-\vb|^\gamma\,\mathbf{S}(v,\vb)\right)=2\,|v-\vb|^\gamma\,(v-\vb).$$
Therefore,
\begin{multline*}
\dis \ddt \ir3 f_\epsilon(t,v)\, \varphi(v) \d v  = - \frac18 \ir3 \ir3 \nabla_v \varphi(\v) \cdot \Big[ f_\epsilon(t,\v)\nabla_{\vb} \cdot \left(|\v-\vb|^\gamma\,\mathbf{S}(v,\vb)\right)\\
 + |v-\vb|^\gamma {\bf S}(\v,\vb)\nabla_v f_\epsilon(t,\v) \Big] \M(\vb) \d v \d \vb + O(\epsilon^2),
\end{multline*}
which, performing the integration with respect to $\vb$ and setting
\begin{equation}
\label{defA}
{\bf D}_\gamma(\v) = \frac{1}{8}\ir3 |v-\vb|^\gamma {\bf S}(\v,\vb)\, \M(\vb)\d \vb\,,
\end{equation} yields
\begin{equation*}\begin{split}\dis \ddt \ir3 f_\epsilon(t,v)\, \varphi(v) \d v  &= - \ir3 \nabla_v \varphi(\v) \cdot \mathbf{D}_\gamma(v)\nabla_v f_\epsilon(t,\v) \d v\\
&\phantom{+++++} -\frac18 \ir3 f_\epsilon(t,v) \nabla_v \varphi(v) \cdot \left(\ir3 \nabla_{\vb} \cdot \left(|\v-\vb|^\gamma\,\mathbf{S}(v,\vb)\right) \M(\vb)\d \vb\right)\\
&=- \frac18\ir3 \nabla_v \varphi(\v) \cdot \mathbf{D}_\gamma(v) \nabla_v f_\epsilon(t,\v) \d v\\
&\phantom{+++} +\frac18 \ir3 f_\epsilon(t,v) \nabla_v \varphi(v) \cdot  \ir3 |v-\vb|^\gamma\,\mathbf{S}(v,\vb) \nabla_{\vb} \M(\vb)\d \vb + O(\epsilon^2). \end{split}\end{equation*}
Since $\nabla_{\vb} M(\vb) = -\, \frac{\vb- u_0}{\theta}\, \mathcal{M}(\vb)$
and  ${\bf S}(\v, \vb) \cdot (\vb - \v) = 0$, we recognize that
$$ \frac18 \ir3 |v-\vb|^\gamma\,\mathbf{S}(v,\vb) \nabla_{\vb} \M(\vb)\d \vb=\frac{1}{\theta}\mathbf{D}_\gamma(v)(v-u_0).$$

Finally, one obtains
\begin{equation}
\label{FPweak}
\ddt \ir3 f_\epsilon(t,v)\, \varphi(v) \d v = \ir3 \varphi(\v) \nabla_{\v} \cdot \left\{ {\bf D}_\gamma(\v)  \left[ \nabla_{\v} f_\epsilon(t,\v) + \frac{\v - u_0}{\theta} f_\epsilon(t,\v) \right] \right\} \d v\,+O(\epsilon^2).
\end{equation}
In particular, one expects the limit $f(t,v)=\lim_{\epsilon \to 0} f_\epsilon(t,v)$ to satisfy the  Fokker--Planck equation
\begin{equation}
\label{FP1}
\p_t f(\v) =  \nabla_{\v} \cdot \left\{ {\bf D}_\gamma(\v) \cdot \left[ \nabla_{\v} f(\v) + \frac{\v - u_0}{\theta} f(\v) \right] \right\}.
\end{equation}
where the diffusion coefficient ${\bf D}_\gamma(\v)$ is defined in \eqref{defA}.

The above computations are clearly formal. Nevertheless, they can be
made rigorous following the lines of \cite{LoTo} (see also
\cite{goudon,desvillettes1992} for similar considerations for the
nonlinear Boltzmann equation) to get the following
\begin{propo}\label{prop:converge}
  Let $f_0 \in L^1(\R^3,(1+|v|^2)\d v)$ be a nonnegative probability
  distribution. For $\gamma \in \{0,1\}$ and $\epsilon \in (0,1)$, let
  $\L_\epsilon$ denote the linear Boltzmann operator with collision
  kernel given by \eqref{eq:Bgammaeps} and let $h_\epsilon(t,\cdot)$
  be the unique solution to \eqref{eq:Leps}. Set
  $f_\epsilon(t,v)=h(t\epsilon^{-2},v)$ for any $t \geq 0$, $v \in
  \R^3$. Then, there exists a subsequence, still denoted
  $(f_\epsilon(t))_\epsilon$ such that
  $$f_\epsilon  \underset{\epsilon \to 0}{\rightharpoonup}
  f \quad \text{ weakly in }  L^1_{\mathrm{loc}}([0,\infty),L^1(\R^3))$$
  where $f=f(t,v)$ is the unique solution to the Fokker-Planck
  equation \eqref{FP1} with initial datum $f(0)=f_0.$
\end{propo}

\subsection{Logarithmic Sobolev inequality}
\label{sec:logS}

We recall here some well-known features about the long-time behavior
of the solution $f(t,v)$ to the Fokker-Planck equation \eqref{FP1}
with initial datum $f(0,v)=f_0$, $f_0 \in L^1(\R^3,(1+|v|^2)\d v)$
being a nonnegative probability distribution.  It is very well known
that the Maxwellian $\M$ is the unique steady state with unit mass to
the Fokker-Planck equation and the convergence of $f(t,v)$ towards
$\M$ (as $t\to \infty$) can be made explicit by the use of entropy
methods (see e.g. \cite{ACJ,Arnold,SCalogero}). Let us explain more in
detail the general strategy (we follow here the introduction of
\cite{SCalogero}). Introduce the change of unknown
$$g(t,v)=\dfrac{f(t,v)}{\M(v)}, \qquad t \geq 0, v \in \R^3$$
Then the relative entropy $\H(f(t)|\M)$ can be rewritten as
$$\H(f(t)|\M)
=\ir3 f(t,v)\log\left(\dfrac{f(t,v)}{\M(v)}\right)\d \v
=\ir3 g(t,v)\log g(t,v) M(v) \d v
$$
where $\d\mu(v)=\M(v)\d v$ is the invariant measure associated to the
Fokker-Planck operator. It is straightforward to check that $g(t,v)$
satisfies now the drift-diffusion equation:
$$\partial_t g(t,v)
= \nabla \cdot \left(\mathbf{D}_\gamma(v) \nabla g(t,v)\right) -
\frac{v-u_0}{\theta} \cdot (\mathbf{D}_\gamma(v)\nabla g(t,v))$$ where
$\nabla=\nabla_v$. One can compute the time derivative of $\H(f|M) =
H(Mg|M)$ by using this to obtain
\begin{equation}
  \label{eq:Igamma}
  \dfrac{\d}{\d t} \H(f(t)|M)
  = -\ir3 \dfrac{\left(\mathbf{D}_\gamma(v) \nabla g(t,v)\right)\cdot
    \nabla g(t,v)}{g(t,v)}\d \mu(v)
  =: -\mathcal{J}_\gamma(f(t)|\M)
  \qquad \forall t \geq 0.
\end{equation}
In particular, $\mathbf{D}_\gamma$ being positive definite, one sees
that $\mathcal{J}_\gamma(f|\M) \geq 0$. If we find $\lambda > 0$ such
that the logarithmic Sobolev inequality
\begin{equation}
  \label{J_gammaM-aim}
  \lambda \H(f|\M) \leq \mathcal{J}_\gamma(f|\M)
\end{equation}
holds for all probability densities $f$, then this can be immediately
used in \eqref{eq:Igamma} to deduce exponential convergence to
equilibrium in the entropy sense for solutions to \eqref{FP1}. If the
diffusion matrix $\mathbf{D}_\gamma$ is the identity then this is the
Gaussian logarithmic Sobolev inequality \eqref{eq:Gaussian-log-Sob},
which holds for $\lambda = 2/\theta$. However, the matrix
$\mathbf{D}_\gamma$ we obtained through the limiting procedure in
Section \ref{sec:grazing} is not the identity; we now consider what
can be said regarding the inequality \eqref{J_gammaM-aim} in the cases
$\gamma = 0$ (Maxwell molecules) and $\gamma = 1$ (hard spheres).

\paragraph{\textit{\textbf{The case $\gamma=0$}}}

For $\gamma=0$ the matrix ${\bf D}_0(v)$ can be explicitly computed,
giving
\begin{equation}
  \label{Aexp}
  \dis {\bf D}_0(\v) =\frac18\left({\bf S}(\v, u_0) + 2 \theta\, {\bf I}\right).
\end{equation}
In particular, since the matrix $\mathbf{S}(\v,u_0)$ is nonnegative,
one sees from the definition \eqref{eq:Igamma} that
\begin{equation*}
  \mathcal{J}_0(f|\M)
  \geq \frac{\theta}{4}\,\ir3 \dfrac{|\nabla g(v)|^2}{g(v)} M(v) \d v
  =
  \frac{\theta}{4} I(f|M)
  \geq
  \frac{1}{2} \H(f|M),
\end{equation*}
for any probability density $f \in L^1(\R^3)$, where we have used the
Stam-Gross inequality \eqref{eq:Gaussian-log-Sob}. That is,
\begin{equation}
  \label{dis:entropyFP1}
  \mathcal{J}_0(f|\M)
  \geq
  \frac{1}{2} \H(f|M),
\end{equation}
which is inequality \eqref{J_gammaM-aim} for $\gamma = 0$ and $\lambda
= 1/2$. From this one deduces that, if $f(t)$ denotes the unique
solution to \eqref{FP1} then
\begin{equation}
  \label{entropyFP1}
  \H(f(t)|\M)
  \leq
  \exp\left(-\frac{1}{2}t\right) \H(f_0|\M)
  \quad \forall t \geq 0.
\end{equation}
We do not know whether $1/2$ is the optimal constant here, since we
have disregarded one of the terms in \eqref{Aexp}.

\medskip
\paragraph{\textbf{\textit{The case $\gamma=1$}}}

For $\gamma=1$, the matrix $\mathbf{D}_1(\v)$ is given by
\begin{equation*}
  \mathbf{D}_1(v)
  =\frac{1}{8}\ir3 \Big[ |v-\vb|^3\, \mathbf{I} - |v - \vb| (v- \vb) \otimes (v -\vb) \Big] \M(\vb)\d\vb
\end{equation*}
With this matrix it is not obvious whether the logarithmic Sobolev inequality
\eqref{J_gammaM-aim} holds for some $\lambda > 0$.

Let us briefly review the Bakry-C	mery criterion for studying this kind
of inequality. We introduce the vector field
$$\mathbf{X} g(v)
:= \mathbf{D}_1(v)\nabla \kappa(v) \nabla g(v), \qquad \text{
  where } \quad \kappa(v) =\sqrt{\mathrm{det}(\mathbf{D}_1(v))}
\exp\left(-\frac{|v-u_0|^2}{2\theta}\right)$$ and the Riemannian
manifold $\Sigma:=(\R^3, \mathbf{D}_1(v)^{-1})$ with
$\mathbf{D}_\gamma(v)^{-1}$ as covariant metric tensor. It is known
\citep{BakryEmery, Bakry, Arnold} that, if there exists some $\alpha >
0$ such that
\begin{equation}
  \label{eq:BEcurva}
  \mathbf{Ric}^\Sigma-\nabla^\Sigma \mathbf{X }\geq \alpha
  \mathbf{D}^{-1}_1
\end{equation}
(where $\mathbf{Ric}^\Sigma$ and $\nabla^\Sigma \mathbf{X}$ denote
respectively the Ricci curvature and the Levi-Civita connection of
$\Sigma$) then it holds:
\begin{equation}
  \label{frakI}
  2\alpha \H(f|M) \leq \mathcal{J}_1(f|M)
\end{equation}
for any nonnegative $f \in L^1(\R^3)$ with unit mass.

In the $\gamma=1$ case the application of the Bakry--\'Emery criterion
is delicate but one can still apply \eqref{eq:BEcurva} to deduce that
for $\alpha = \frac{7}{24} \sqrt{\frac{2 \theta}{\pi}}$ we have
\begin{equation}
  \label{dis:entropyFP11}
  \mathcal{J}_1(f|\M) \geq 2\alpha \H(f|\M)
\end{equation}
for any probability density $f \in L^1(\R^3).$ Details are reported in
Appendix~\ref{ApB}.

\subsection{Entropy dissipation for grazing collisions kernels}

We consider here the entropy dissipation functionals associated to
grazing collision kernels as introduced in Section
\ref{sec:grazing}. For simplicity we discuss only the case of
dimension $d=3$, though our analysis can be extended (with more
cumbersome calculations) to $d \geq 2$. We consider here \emph{grazing
  hard-spheres collision kernels} of the form
$B_\epsilon(|q|,\xi)=\beta(|q|) b_\epsilon(\xi)$ with
\begin{gather}
  \label{eq:beta-hardsphere}
  \beta(|q|) = |q|,
  \\
  \label{eq:b-diffusion-limit}
  b_\epsilon(\xi) = \frac{\xi}{\|b_\epsilon\|} \1_{[0,\epsilon]}(\xi),
  \qquad \|b_\epsilon\|:=|\S| \int_{0}^\epsilon \xi  (1-\xi^2)^{\frac{d-3}{2}}\d \xi,
\end{gather}
for some $\epsilon \in (0,1]$. Notice that the case $\epsilon=1$
corresponds to hard-spheres interactions for which $B(|\q|,\xi) =
\|b_1\|^{-1} |q\cdot n|.$ We show now that for such collision kernels
Assumption \ref{convol} is met, thus extending Proposition
\ref{prp:D-comparison-hardpotentials} to include them:
\begin{propo}[Comparison of dissipations for grazing hard-spheres]
  \label{prp:D-comparison}
  Take $\epsilon > 0$ and let
  $B_\epsilon(|q|,\xi)= \beta(|q|) b_\epsilon(\xi)$ be the grazing
  hard-spheres kernel given by
  \eqref{eq:beta-hardsphere}--\eqref{eq:b-diffusion-limit}, in
  dimension $d=3$. We denote by
  $\tilde{B}_\epsilon(\xi)=b_\epsilon(\xi)$ the associated normalized
  Maxwellian collision kernel. There exists some number $C > 0$
  (\emph{independent} of $\epsilon$) such that
  \begin{equation}
    \label{eq:kernel-comparison}
    D_\Phi^{B_\epsilon}(f) \geq C D_\Phi^{\tilde{B}_\epsilon}(f)
  \end{equation}
  for any convex function $\Phi\::\:[0,\infty) \to [0,\infty)$ (where
  $D_\Phi$ is defined in Section \ref{sec:non-maxwell}).
\end{propo}

\begin{proof}
  As for Proposition \ref{prp:D-comparison-hardpotentials}, the proof
  consists in checking that Assumption \ref{convol} is met by the
  kernels $\beta(|q|)$ and $b_\epsilon(\xi).$ First, for a given
  $\epsilon > 0$, one sets
$$\delta_\epsilon:=\dfrac{\sqrt{1-\epsilon^2}}{\epsilon}.$$
For a given $s > 0$ and a given $\bar{v} \in \R^{2}$, the numerator of \eqref{convol} reads
\begin{equation*}
  \int_{\R^{2}}
  \beta\left(\sqrt{|\bar{v}-\bar{v}_*|^2+s^2}\right)
  b_\epsilon\left(\frac{s}{\sqrt{|\bar{v}-\bar{v}_*|^2+s^2}}\right)
  \, \M(\bar{v}_*) \, \d \bar{v}_*
  =
  \frac{s}{\|b_\epsilon\|}\int_{\R^{2}}
  \1_{\{|\bar{v}-\bar{v}_*| \geq s\delta_\epsilon\} }\,
  \M(\bar{v}_*)\,\d \bar{v}_*
\end{equation*}
while the denominator is simply
$$\frac{s}{\|b_\epsilon\|}\int_{\R^{2}} \dfrac{\M(\bar{v}_*)}{\sqrt{|\bar{v}-\bar{v}_*|^2+ s^2}}\1_{\{|\bar{v}-\bar{v}_*| \geq s\delta_\epsilon\}}\,\,\d \bar{v}_*.$$
Hence to prove \eqref{convol}, it suffices clearly to show that there
exists $C >0$ such that
$$\int_{\{|\bar{v}-\bar{v}_*| \geq s\delta_\epsilon\}} \M(\bar{v}_*)\,\d \bar{v}_* \geq C\int_{\{|\bar{v}-\bar{v}_*| \geq s\delta_\epsilon\}} \dfrac{\M(\bar{v}_*)}{\sqrt{|\bar{v}-\bar{v}_*|^2+s^2}}\,\d \bar{v}_*$$
for any $\bar{v} \in \R^{2}$ and any $s > 0.$  If $s\delta_\epsilon \geq 1 $ we can directly
  bound $|\bar{v}-\bar{v}_*| \geq 1$ and the inequality is obviously
  true with $C = 1$. If $s\delta_\epsilon < 1$ we have
  \begin{equation*}
 \int_{\{|\bar{v}-\bar{v}_*| \geq s\delta_\epsilon\}} \M(\bar{v}_*)\,\d \bar{v}_* \geq
    \int_{|\bar{v}_*-\bar{u}_0| \geq 1}
    \M(\bar{v}_*) \d \bar{v}_*
    =: C_1 > 0
  \end{equation*}
  while
  \begin{equation*}
    \int_{\{|\bar{v}-\bar{v}_*| \geq s\delta_\epsilon\}}
    \dfrac{\M(\bar{v}_*)}
    {\sqrt{|\bar{v}-\bar{v}_*|^2+s^2}}\,\d \bar{v}_*
    \leq
    \int_{\R^{2}} \dfrac{\M(\bar{v}_*)}{|\bar{v}-\bar{v}_*|}
    \,\d \bar{v}_* \leq C_2
    < \infty.
  \end{equation*}
  Therefore, when $s \delta_\epsilon < 1$, the result holds with $C=
  C_1/C_2$. This shows that \eqref{convol} holds true with
  $\tilde{C}_\theta = \max\{1,\frac{C_1}{C_2}\}$ independent of
  $\epsilon$ (and also of $\varrho_0$ in this case).
\end{proof}

Let us explain now, in a rather informal way, how the above result
together with Theorem \ref{theo:max} allows us to use
\eqref{entropydiss} to give a proof of the logarithmic Sobolev
inequality \eqref{J_gammaM-aim} for $\gamma = 0$ or $\gamma = 1$.

\medskip
\paragraph{\textbf{\textit{The case $\gamma=0$}}}

Since $\tilde{B}_\epsilon$ is a normalized Maxwellian collision
kernel, for the special choice of the convex function $\Phi(x)=x\log x
-x +1$, if we denote simply by $\D_{\mathrm{max},\epsilon}$ the
associated entropy dissipation functional, Theorem \ref{theo:max}
asserts that
$$\D_{\mathrm{max},\epsilon}(f) \geq \gamma_\epsilon \H(f|\M)$$
for any probability distribution $f \in L^1(\R^d)$ with $\gamma_\epsilon:=\gamma_{b_\epsilon}=\dfrac{\int_0^\epsilon \xi^3\left(1-\xi^2\right)^{\frac{d-3}{2}}\d\xi}{\int_0^\epsilon \xi\left(1-\xi^2\right)^{\frac{d-3}{2}}\d\xi}.$
In particular, in dimension $d=3$, one has
$$\gamma_\epsilon=\frac{\epsilon^2}{2}.$$
Therefore, in dimension $d=3$, the solution $h_\epsilon(t,\cdot)$ to
\eqref{eq:Leps} with $\gamma = 0$ satisfies
$$\H(h_\epsilon(t)|\M)
\leq \exp\left(-\frac{\epsilon^2}{2} t\right) \H(f_0|\M)
\qquad \forall t \geq 0$$
which, in terms of the rescaled function $f_\epsilon(t,\cdot)$ defined in \eqref{scaling} reads
$$\H(f_\epsilon(t)|\M) \leq \exp\left(-\frac{t}{2}  \right) \H(f_0|\M) \qquad \forall t \geq 0.$$
In particular, from Proposition \ref{prop:converge}, taking the limit as $\epsilon \to 0$, we get that
$$\H(f(t)|\M) \leq \exp\left(-\frac{t}{2}\right) \H(f_0|\M) \qquad \forall t \geq 0$$
for any solution $f(t)$ to the Fokker-Planck equation \eqref{FP1}
(associated to the diffusion matrix $\mathbf{D}_0(\cdot)$). We recover
in this way \eqref{entropyFP1} which, as well-known \citep{Arnold}, is
equivalent to \eqref{dis:entropyFP1}, i.e.,
\begin{equation*}
  \mathcal{J}_0(f|\M) \geq \frac12 \H(f|\M).
\end{equation*}
This shows that the log-Sobolev inequality \eqref{dis:entropyFP1} can
be recovered from \eqref{entropydiss} in the limit of grazing
collisions and, as explained in the introduction, this can be seen as
providing a microscopic ground for these particular log-Sobolev
inequalities.

\medskip
\paragraph{\textbf{\textit{The case $\gamma=1$}}}

In the same way, if one considers now grazing collisions for
hard-spheres interactions $B_\epsilon(|q|,\xi)$, the same reasoning
can be applied and using Proposition \ref{prp:D-comparison} we see
that any solution $h_\epsilon(t,\cdot)$ to \eqref{eq:Leps} satisfies
$$\H(h_\epsilon(t)|\M) \leq \exp\left(-C\frac{\epsilon^2}{2} t\right) \H(f_0|\M) \qquad \forall t \geq 0$$
for some explicitly computable constant $C > 0$. In terms of the rescaled function $f_\epsilon$ we get now
$$\H(f_\epsilon(t)|\M) \leq \exp\left(-\frac{C}{2} t\right)\H(f_0|\M) \qquad \forall t  \geq 0$$
which, at the limit $\epsilon \to 0$, yields now
$$\H(f(t)|\M) \leq \exp\left(-\frac{C}{2} t\right) \H(f_0|\M) \qquad \forall t \geq 0$$
for any solution $f(t)$ to the Fokker-Planck equation \eqref{FP1}
(associated to the diffusion matrix $\mathbf{D}_1(\cdot)$). In
particular, one deduces from this inequality that the functional
inequality
$$\mathcal{J}_1(f|\M) \geq \frac{C}{2}H(f|\M)$$
holds true for any probability density $f \in L^1(\R^3)$ with finite
entropy. Again, in the grazing collisions limit, Theorem \ref{main}
allows us to prove the nontrivial log-Sobolev inequality
\eqref{J_gammaM-aim} for $\gamma = 1$ --- with a not necessarily
optimal constant.

\appendix

\section{Proof of Theorem \ref{theo12}}
\label{sec:no-log-sob}

We give in this Appendix a quick proof of Theorem \ref{theo12}, which
states that the linear Boltzmann operator $\L$ does not satisfy the
log-Sobolev inequality \eqref{eq:log-sob-L}. That is, if we define the
quadratic form associated to $\L$ as
\begin{equation*}
  \mathscr{E} (g) := -\ird \frac{1}{\M(v)}\, g(v) \L g(v) \,d v,
\end{equation*}
we show it is not possible to find a positive $\lambda_0 > 0$ such
that
\begin{equation}
  \label{eq:log-sob-LApp}
  \mathscr{E}\left( \sqrt{\M} \sqrt{f} \right)
  \geq \lambda_0 \H(f|\M)
\end{equation}
for any probability density $f \in L^1(\R^d,\d v)$. The proof is based
on the well-known fact that \eqref{eq:log-sob-LApp} is equivalent to
Nelson's hypercontractivity. In order to apply directly Nelson's hypercontractivity, one shall reformulate the problem in some equivalent way to define the Markov semigroup associated to $\L$. Namely, we introduce the probability measure
$$\d\mu(v)=\M(v)\d v$$
and the Markov operator
$$\mathbf{L}(h)=\M^{-1}\L(h\,\M), \qquad \forall h \in \mathscr{D}(\mathbf{L})$$
where $\mathscr{D}(\mathbf{L})$ denotes the domain of $\mathbf{L}$ in the space $L^{2}(\R^{d},\d\mu)$. Notice that $\mathbf{L}$ is a Markov operator in the sense of \cite{Bakry2014} since
$$\int_{\R^{d}}\mathbf{L}h(v)\d\mu(v)=0 \qquad \text{ and } \qquad \mathbf{L}(1)=0.$$
Notice that,  with such notations and using the terminology of \citet{Ane, Bakry2014}, it holds
$$\H(f|\M)=\mathbf{Ent}_{\mu}\left(fM^{-1}\right), \qquad \mathscr{E}(g)=-\int_{\R^{d}}\dfrac{g}{\M}\mathbf{L}\left(\dfrac{g}{\M}\right)\d\mu=:\mathcal{E}_{\mu}(gM^{-1})$$
so that \eqref{eq:log-sob-LApp} reads equivalently
$$\mathcal{E}_{\mu}(\sqrt{h}) \geq \lambda_{0}\mathbf{Ent}_{\mu}(h), \qquad h=f\M^{-1}$$
which is the classical Log-Sobolev inequality for the measure $\mu$
and the associated Dirichlet form $\mathcal{E}_{\mu}$ (see
\cite{Ane,Bakry2014} for details). Considering then the Markov
semigroup $(\mathcal{S}_t)_t$ generated by $\mathbf{L}$, we recall the
following result (see \citet[Theorem 2.8.2]{Ane}):
\begin{lemme}[\cite{Gross1975Logarithmic}] If
  \eqref{eq:log-sob-LApp} holds true with $\lambda_0 > 0$ then
  $$\|\mathcal{S}_t h\|_{L^{q(t)}(\R^d,\d\mu)} \leq \|h\|_{L^2(\R^d,\d\mu)} \qquad \forall t > 0$$
where $q(t)=1+\exp(4t/\lambda_0)$ for any $t \geq 0$ and $\d\mu(v)=\M(v)\d v$ is the invariant measure associated to $\L$.
\end{lemme}
\begin{nb} Notice that the above Lemma is valid for any Markov
  semigroup whose invariant measure is reversible. This is the case
  for the semigroup $(\mathcal{S}_t)_t$ associated to $\mathbf{L}$ by
  virtue of the detailed balance principle
  \eqref{eq:detailed-balance}.
\end{nb}
In particular, if \eqref{eq:log-sob-LApp} holds true, then, for $h_0$
in $L^2(\R^d,\d v)$ and for some $t > 0$, there exists $p > 2$ such
that
\begin{equation}
  \label{eq:regul}
  \|h(t)\|_{L^p(\R^d,\d\mu)} \leq \|h_0\|_{L^2(\R^d,\d\mu)}
\end{equation}
where $h(t)=\mathcal{S}_t h_0$ is the unique solution to $\partial_t
h(t) = \mathbf{L}(h)$ with initial condition $h_0$ (in other words,
$f(t) = M h(t)$ is the unique solution to
eq.~\eqref{eq:linear-Boltzmann} with initial data $M h_0$).

We show that such a $L^2-L^p$ regularizing property of
$(\mathcal{S}_t)_t$ cannot hold.  Using the representation
\eqref{eq:linear-Boltzmann-kernel-form}, one sees that
\begin{equation*} \L f(v) \geq -\sigma(v) f(v) \qquad
  \forall f \geq 0
\end{equation*}
where $\sigma(v)$ is the collision frequency (depending on the
collision kernel $B(|q|,\xi)$) \footnote{We recall that if
  $B(|q|,\xi)=|q\cdot n|$ (corresponding to hard-spheres interactions)
  then $\sigma(v) \geq c(1+|v|)$ for some $c > 0$; while if
  $B(|q|,\xi)=c_d \xi$ (corresponding to a normalized Maxwellian
  collision kernel) then $\sigma(v)=1$ for any $v \in \R^d.$}. This
translates obviously into
\begin{equation}
  \label{compareLsigma}
  \mathbf{L} h(v) \geq -\sigma(v) h(v) \qquad \forall h \geq 0.\end{equation}
Let us now consider $h_0 \in L^2(\R^d,\d\mu)$ nonnegative and let $h(t)=\mathcal{S}_t h_0$ while $g(t,v)$ denotes the unique solution to
$$\partial_t g(t,v)=-\sigma(v)g(t,v) \qquad  g(t=0,v)=h_0(v).$$
Clearly
\begin{equation}\label{eq:gt}g(t,v)=\exp(-\sigma(v) t)h_0(v)\end{equation}
and \eqref{compareLsigma} implies that
$$h(t,v) \geq g(t,v) \qquad \forall t \geq 0.$$
Now, it is clear from \eqref{eq:gt} that the above equation for
$g(t,v)$ has no regularizing effect; i.e., if
$h_0 \notin L^p(\R^d,\d\mu)$ then $g(t,v) \notin L^p(\R^d,\d\mu)$ for
any $t \geq 0.$ One deduces from this that, if
$h_0 \notin L^p(\R^d,\d\mu)$ then $h(t,v) \notin L^p
(\R^d,\d\mu)$.
Therefore, inequality \eqref{eq:regul} cannot hold true for any
$h_0 \in L^2(\R^d,\d\mu)$ and Gross' Theorem shows that
\eqref{eq:log-sob-LApp} cannot hold true.

\section{Bakry-\'Emery criterion for hard-spheres interactions}
\label{ApB}

We give here a direct proof of the logarithmic Sobolev inequality
\eqref{J_gammaM-aim} in the case $\gamma=1$ using the Bakry-C	mery
criterion (see \eqref{eq:BEcurva}).
We use the notation of Section \ref{sec:logS}. In the hard-spheres
case ($\gamma = 1$) the diffusion matrix of the associated
Fokker--Planck equation reads as
\begin{equation}
\mathbf{D}_1(v)=\frac{1}{8}\ir3 \Big[ |v-\vb|^3\, \mathbf{I} - |v - \vb| (v- \vb) \otimes (v -\vb) \Big] \M(\vb)\d\vb
\end{equation}
Since $\M(\vb) = \left( \frac{1}{2 \pi \theta} \right)^{3/2} \exp \left( - \frac{|\vb - u_0|^2}{2 \theta} \right)$, the diffusion matrix may be cast as
$$
\mathbf{D}_1(v) = \frac{\gamma^{3/2}}{8\, \pi^{3/2}}\, \ir3 \Big[ |w|^3\, \mathbf{I} - |w| w \otimes w \Big] \exp\left(-\gamma |w+a|^2\right)\d w
$$
where
$$
\gamma= \frac{1}{2 \theta}\,, \qquad \qquad a= u_0 - v\,.
$$
It can be directly checked (see~\cite{BisiSpigaJMP}) that
\begin{equation}
\dis \ir3 |w|^3\, {\rm e}^{-\gamma |w+a|^2} \d w = \frac{\pi}{\gamma^3} \left\{ {\rm e}^{- \gamma |a|^2} \left( \gamma |a|^2 + \frac52 \right)
+ \frac{\sqrt{\pi}\, {\rm erf} \left( \gamma^{1/2} |a| \right)}{\gamma^{1/2} |a|} \left( \gamma^2 |a|^4 + 3 \gamma |a|^2 + \frac34 \right) \right\}
\end{equation}
and, analogously,
\begin{multline*}
  \dis \ir3 |w| (w \otimes w) {\rm e}^{-\gamma |w+a|^2} \d w
  \\
  = \frac{\pi}{\gamma^3}\, {\bf I}\, \left\{ {\rm e}^{- \gamma |a|^2} \left( \frac12 + \frac{1}{4 \gamma |a|^2} \right)
    + \frac{\sqrt{\pi}\, {\rm erf} \left( \gamma^{1/2} |a|
      \right)}{\gamma^{1/2} |a|} \left( \frac12\, \gamma |a|^2 +
      \frac12 - \frac{1}{8\, \gamma |a|^2} \right) \right\}
  \\
  \dis + \frac{\pi}{\gamma^3}\, \frac{a \otimes a}{|a|^2} \left\{ {\rm e}^{- \gamma |a|^2} \left( \gamma |a|^2 + 1 - \frac{3}{4 \gamma |a|^2} \right)
    + \frac{\sqrt{\pi}\, {\rm erf} \left( \gamma^{1/2} |a| \right)}{\gamma^{1/2} |a|} \left( \gamma^2 |a|^4 + \frac34\, \gamma |a|^2 - \frac34 + \frac{3}{8\, \gamma |a|^2} \right) \right\}
\end{multline*}
where erf denotes the error function
$\mathrm{erf}(x)=\frac{2}{\sqrt{\pi}}\ds\int_{0}^{x} {\rm
  e}^{-t^{2}}\d t$, $x \geq 0$.  Consequently,
\begin{equation}
\mathbf{D}_1(v) = \frac{1}{8 \sqrt{\pi}\, \gamma^{3/2}}\, \left\{\mathrm{C}\left( \gamma^{1/2} |a| \right)\, {\bf I} +
\mathrm{T}\left( \gamma^{1/2} |a| \right) \frac{|a|^2 {\bf I} - a \otimes a}{|a|^2} \right\}
\label{D1comp}
\end{equation}
where
$$
\mathrm{C}(|x|) = {\rm e}^{-\, |x|^2} \left( 1 + \frac{1}{2 |x|^2} \right) + \frac{\sqrt{\pi}\, {\rm erf}(|x|)}{|x|} \left( \frac74\, |x|^2 + 1 + \frac{1}{4 |x|^2} \right)
$$
and
$$
\mathrm{T}(|x|) = {\rm e}^{-\, |x|^2} \left( |x|^2 + 1 - \frac{3}{4 |x|^2} \right) + \frac{\sqrt{\pi}\, {\rm erf}(|x|)}{|x|} \left( |x|^4 + \frac34\, |x|^2 - \frac34 + \frac{3}{8 |x|^2} \right).
$$
By resorting also to Taylor expansions
$$
{\rm e}^{-\, |x|^2} = 1 - |x|^2 + \frac{\xi^4}{2}\,, \qquad \quad
{\rm erf}(|x|) = \frac{2}{\sqrt{\pi}} \left( |x| - \frac{|x|^3}{3} + \frac{\eta^5}{10} \right)
$$
(for suitable $\xi, \eta \in [0,|x|]$), it can be checked that
$\mathrm{C}(|x|)>0$ and $\mathrm{T} (|x|) \geq 0$ for any~$x$.
Consequently, the matrix $T \left( \gamma^{1/2} |a| \right)
\frac{|a|^2 {\bf I} - a \otimes a}{|a|^2}$ appearing in~
\eqref{D1comp} is positive definite and to derive an estimate like
\eqref{frakI} one can neglect its contribution and consider only the
diffusion matrix
$$\frac{1}{8 \sqrt{\pi}\, \gamma^{3/2}}\, d(v)\mathbf{I}:= \frac{1}{8
  \sqrt{\pi}\, \gamma^{3/2}}\, \mathrm{C}\left( \gamma^{1/2} |a|
\right) {\bf I}.$$ In this case, the Bakry-\'Emery curvature condition
\eqref{eq:BEcurva} simply reads, writing $E(v) :=
\frac{|v-u_0|^2}{2\theta}$, (see \cite{Bakry, Arnold}):
\begin{multline}\label{eq:BEcurvaIde}
  -\frac{1}{4} \dfrac{\nabla d(v)\otimes \nabla d(v)}{d(v)}
  + \frac{1}{2}\left(\Delta d(v)
    - \nabla d(v) \cdot \nabla E(v)\right)\mathbf{I}
  \\
  + d(v)\mathbb{D}^2 E(v)
  + \dfrac{\nabla d(v) \otimes \nabla E(v)
    + \nabla E(v) \otimes \nabla d(v)}{2}
  -\mathbb{D}^2 d(v)
  \geq \alpha\, 8 \sqrt{\pi}\, \gamma^{3/2}\, \mathbf{I}
\end{multline}
(in the sense of positive matrices) for any $v \in \R^3$, where we
recall that $\mathbb{D}^2$ denotes the Hessian matrix while
$E(v)=\frac{|v-u_0|^2}{2 \theta}$.  Since
$$
\begin{array}{c}
\dis \nabla d(v) = \mathrm{C}'\left( \gamma^{1/2} |a| \right) \gamma^{1/2} \frac{v-u_0}{|v-u_0|}\,, \qquad \qquad
\Delta d(v) = \mathrm{C}''\left( \gamma^{1/2} |a| \right) \gamma\,, \vspace*{0.3 cm}\\
\dis \mathbb{D}^2 d(v) =\mathrm{C}''\left( \gamma^{1/2} |a| \right) \gamma\, \frac{(v-u_0) \otimes (v-u_0)}{|v-u_0|^2} + \mathrm{C}'\left( \gamma^{1/2} |a| \right) \gamma^{1/2} \left[ \frac{1}{|v-u_0|}\, {\bf I} -  \frac{(v-u_0) \otimes (v-u_0)}{|v-u_0|^3} \right],
\end{array}
$$
formula (\ref{eq:BEcurvaIde}) becomes
\begin{equation}
  A\left(
    \gamma^{1/2} |a| \right)\, {\bf I} + B\left( \gamma^{1/2} |a|
  \right)
  \sum_{i,j=1}^3 \frac{y_i y_j}{|y|^2}\,
  \frac{(v_i - u_{0i}) (v_j - u_{0j})}{|v-u_0|^2}
  \geq \alpha\, 8 \sqrt{\pi\, \gamma}
\end{equation}
for $v \in \R^3$, $y \in \R^3 \setminus \{ 0 \}$, where
\begin{gather*}
  A(|x|)
  = \frac12\, \mathrm{C}''(|x|)
  - \left( |x| + \frac{1}{|x|} \right) \mathrm{C}'(|x|)
  + 2\, \mathrm{C}(|x|)\,,
  \\
  B(|x|) = \mathrm{C}''(|x|) - \left( 2\, |x| + \frac{1}{|x|} \right)
  \mathrm{C}'(|x|)
  + \frac14\, \frac{\left[\mathrm{C}'(|x|) \right]^2}{\mathrm{C}(|x|)}.
\end{gather*}
Since
$$
0 \leq \sum_{i,j=1}^3 \frac{y_i y_j}{|y|^2}\, \frac{(v_i - u_{0i}) (v_j - u_{0j})}{|v-u_0|^2} \leq 1\,,
$$
a (non--optimal) estimate for $\alpha$ is provided by $\alpha = \frac{1}{8 \sqrt{\pi\, \gamma}} \min \{ \alpha_1, \alpha_2 \}$ with $\alpha_1$, $\alpha_2$ such that
\begin{equation}
A(|x|) \geq \alpha_1\,, \qquad \quad A(|x|) - B(|x|) \geq \alpha_2\,, \qquad \quad \forall x \in \R^3.
\end{equation}
Now we have
$$
A(|x|) = {\rm e}^{-\, |x|^2} \left( -\,3\,|x|^2 + 1 + \frac{3}{2 |x|^2} + \frac{9}{2 |x|^4} \right) + \frac{\sqrt{\pi}\, {\rm erf}(|x|)}{|x|} \left( \frac74\, |x|^2 + \frac54 + \frac{3}{4 |x|^2} - \frac{9}{4 |x|^4} \right);
$$
by using Taylor expansion (for $|x| \leq 1$) and by studying the derivative $A'(|x|)$ (for higher~$|x|$) we get that $A(|x|) \geq \alpha_1= \frac{143}{60} \simeq 2.38$. On the other hand, for $A(|x|) - B(|x|)$ we use the estimate
$$
\left| \frac{\mathrm{C}'(|x|)}{\mathrm{C}(|x|)} \right| \leq \min \left\{ 1, \frac{1}{|x|} \right\}\,,
$$
and we check that for $|x| > 1$ the quantity $A(|x|) - B(|x|)$ turns out be bounded from below by a constant greater than~$\alpha_1$ (so we skip details here) while, for $|x| \leq 1$,
$$
A(|x|) - B(|x|) \geq {\rm e}^{-\, |x|^2} \left( 3\,|x|^2 + \frac54 - \frac{5}{8 |x|^2} - \frac{3}{|x|^4} \right) + \frac{\sqrt{\pi}\, {\rm erf}(|x|)}{|x|} \left( \frac{41}{16}\, |x|^2 + \frac34 + \frac{11}{16 |x|^2} + \frac{3}{2 |x|^4} \right)
$$
that by Taylor expansion turns out to be greater than $\alpha_2= \frac73$. In conclusion, $\alpha = \frac{7}{24} \sqrt{\frac{2 \theta}{\pi}}$.

\subsection*{Acknowledgments} 
M.~B.~acknowledges support of Italian GNFM and of the University of
Parma.  J.~A.~C.~was supported by the Marie-Curie CIG grant KineticCF
and the Spanish project MTM2011-27739-C04-02. B.~L.~acknowledges
support of the \emph{de Castro Statistics Initiative}, Collegio
C. Alberto, Moncalieri, Italy.



\end{document}